\documentclass[11pt,a4paper]{article}
\title{\bf Quantitative stability estimates for Fokker--Planck equations}
\author{Huaiqian Li\footnote{Email: hqlee@scu.edu.cn. School of Mathematics, Sichuan University, Chengdu 610064, P. R. China.}  \quad Dejun Luo\footnote{Email: luodj@amss.ac.cn. Key Laboratory of Random Complex Structures and Data Sciences, Academy of Mathematics and Systems Science, Chinese Academy of Sciences, Beijing 100190, China and School of Mathematical Sciences, University of the Chinese Academy of Sciences, Beijing 100049, China.} } \vspace{2mm}

\date{}
\usepackage{amssymb,amsmath,amsfonts,amsthm,color,mathrsfs}
\usepackage{esint}

\def\R{\mathbb{R}}
\def\E{\mathbb{E}}
\def\P{\mathbb{P}}
\def\D{\mathcal{D}}

\def\L{\mathcal{L}}

\def\W{\mathbb{W}}
\def\d{\textup{d}}
\def\eps{\varepsilon}

\def\det{\textup{det}}

\def\supp{\textup{supp}}
\def\<{\langle}
\def\>{\rangle}

\def\F{\mathcal{F}}

\def\div{\textup{div}}

\def\newdot{{\kern.8pt\cdot\kern.8pt}}

\newtheorem{theorem}{Theorem}[section]
\newtheorem{lemma}[theorem]{Lemma}
\newtheorem{corollary}[theorem]{Corollary}
\newtheorem{proposition}[theorem]{Proposition}

\newtheorem{definition}[theorem]{Definition}
\theoremstyle{definition}\newtheorem{remark}[theorem]{Remark}

\begin{document}

\maketitle
\makeatletter 
\renewcommand\theequation{\thesection.\arabic{equation}}
\@addtoreset{equation}{section}
\makeatother 

\vspace{-6mm}

\begin{abstract}
We consider the Fokker--Planck equations  with irregular coefficients. Two different cases are treated: in the degenerate case, the coefficients are assumed to be weakly differentiable, while in the non-degenerate case the drift satisfies only the Ladyzhenskaya--Prodi--Serrin condition. Using Trevisan's superposition principle which represents the solution as the marginal of the solution to the martingale problem of the diffusion operator, we establish quantitative stability estimates for the solutions of Fokker--Planck equations.
\end{abstract}



\section{Introduction}

Fix $T>0$ and let $\mathcal{P}(\R^d)$ be the class of  probability measures on the Euclidean space $\R^d$. Let $a:[0,T]\times\R^d\to \mathcal M_{d,d}$ and $b:[0,T]\times\R^d \to\R^d$ be measurable functions, where $\mathcal M_{n,m}$ is the space of $n\times m$ matrices. We consider the possibly degenerate Fokker--Planck equation in $[0,T]\times\R^d$:
  \begin{equation}\label{FPE}
  \partial_t\mu_t-\frac12\sum_{ij}\partial_{ij} (\mu_t a_{ij} )+\div(\mu_t b)=0,\quad \mu_0=\nu,
  \end{equation}
where $\nu\in \mathcal{P}(\R^d)$. A Borel curve $\mu=(\mu_t)_{t\in [0,T]}\subset \mathcal{P}(\R^d)$ is called a weak solution of \eqref{FPE} if
  \begin{equation}\label{integrability}
  \int_0^T\!\!\int_{\R^d} \big(\|a_t\|+|b_t|\big)\,\d\mu_t\d t<\infty
  \end{equation}
and for any $f\in C_c^{1,2}([0,T)\times\R^d)$, one has
  \begin{equation}\label{FPE-1}
  \int_{\R^d} f(0,x)\,\d\nu(x)+ \int_0^T\!\! \int_{\R^d} \big(\partial_t f(t,x)+\L f(t,x)\big)\,\d\mu_t(x)\d t=0,
  \end{equation}
where $\|\cdot\|$ is the Hilbert--Schmidt norm of matrices and $\L$ is the time-dependent second order differential operator associated to \eqref{FPE}. As remarked in \cite[Remark 2.3]{Trevisan}, any solution $(\mu_t)_{t\in [0,T]}$ to \eqref{FPE} admits a unique narrowly continuous representative $(\tilde\mu_t)_{t\in [0,T]}$; hence it is reasonable to say that the solution starts from $\nu$. Since the coefficients are not continuous, the integral in \eqref{FPE-1} would be sensitive to the choice of equivalent versions when $\mu_t$ is singular to the Lebesgue measure. Therefore, we shall assume that, for a.e. $t\in (0,T)$, $\mu_t$ is absolutely continuous with respect to the Lebesgue measure $\d x$ and do not distinguish $\mu_t$ from its density $u_t\in L^1(\R^d, \R_+)$.

In this paper, we assume that there exists a matrix-valued function $\sigma:[0,T]\times\R^d\to \mathcal M_{d,m}$ such that $a=\sigma\sigma^\ast$. When $\sigma$ and $b$ are sufficiently smooth, for example, $\sigma, b\in C^{0,1}_b([0,T]\times\R^d)$, it is well known that the solution $\mu_t$ of the Fokker--Planck equation \eqref{FPE} is the distribution of the solution $X_t$ to the stochastic differential equation
  \begin{equation}\label{Ito-SDE}
  \d X_t=\sigma_t(X_t)\,\d B_t+ b_t(X_t)\,\d t,\quad \mbox{law}(X_0)=\nu,
  \end{equation}
where $(B_t)_{t\geq 0}$ is an $m$-dimensional standard Brownian motion.

If the diffusion coefficient $a$ is identically zero, then the Fokker--Planck equation \eqref{FPE} reduces to the continuity equation
  \begin{equation}\label{CE}
  \partial_t\mu_t+\div(\mu_t b)=0,\quad \mu_0=\nu.
  \end{equation}
According to the celebrated DiPerna--Lions theory, the well-posedness of \eqref{CE} implies the existence and uniqueness of measurable flows $(X_t)_{0\leq t\leq T}$ generated by the ordinary differential equation
  \begin{equation}\label{ODE}
  \d X_t=b_t(X_t)\, \d t,
  \end{equation}
even though the driving vector field $b$ only enjoys some mild regularity properties. In the past three decades, there have been lots of studies on \eqref{CE} (or the transport equation) under various types of assumptions on the regularity of $b$, among which we mention the ground breaking works \cite{DiPernaLions, Ambrosio04}, where the Sobolev and the BV spatial regularities on $b$ were considered, respectively. There are also stability results on the equation \eqref{CE}: if the sequence of vector fields $b^n$ converge to some $b$ in a certain sense, then the corresponding solutions $\mu^n$ tend to $\mu$ too (see for instance \cite[Theorem II.4]{DiPernaLions}). However, there is no explicit quantitative estimate on the convergence rate. These results have been extended to the Fokker--Planck type equations \eqref{FPE}; see \cite{LebrisLions04, LebrisLions08, Figalli, RocknerZhang10, Luo13, Luo14} for the finite dimensional case and \cite{Luo10} for  the results in the abstract Wiener space. These papers deal mainly with the existence and uniqueness of solutions to \eqref{FPE}. Though the stability of solutions has been treated in \cite[Theorem 3.1]{LebrisLions04} (see also \cite[Theorem 1.5]{Ambrosio09} in the setting of a separable Hilbert space endowed with a log-concave measure), there is no explicit estimate on the rate of convergence. We mention that some sufficient conditions were provided in \cite{BDRS} which guarantee the uniqueness of \eqref{CE} in the class of measures, see Theorem 1.1 therein for precise statements. The readers can find in \cite[Section 2]{BRS} a review of the existing methods for proving uniqueness of \eqref{FPE}, and some uniqueness results in the subsequent sections. We also refer to the monograph \cite{Bogachev15} for a comprehensive study of the Fokker--Planck equation \eqref{FPE}.

In the smooth situation, however, some simple calculations lead to the quantitative estimates on the solutions to continuity equations \eqref{CE} (see also \cite[Remark 1.6]{Bogachev16}). Suppose we are given two vector fields $b^1$ and $b^2$, both satisfying the Lipschitz condition with the same constant $L>0$. For simplification of notations, we assume they are time-independent. Let $\mu^i_t$ be the solution to \eqref{CE} with $b=b^i$ and the same initial value $\mu^i_0=\nu\in\mathcal{P}(\R^d)$, $i=1,2$.
Let $p\geq1$; we are concerned with the estimate on the classical $p$-Kantorovich--Wasserstein distance $W_p(\mu_t^1, \mu_t^2)$ between $\mu^1_t$ and $\mu^2_t$. Recall that for $\mu, \nu \in\mathcal{P}(\R^d)$ with finite moment of order $p$,
  $$W_p(\mu,\nu)=\bigg[ \inf_{\pi\in \mathcal C(\mu,\nu)} \int_{\R^d \times\R^d} |x-y|^p\, \d\pi(x,y)\bigg]^{1/p},$$
where $\mathcal C(\mu,\nu)$ is the collection of probability measures on $\R^d\times \R^d$ which have $\mu$ and $\nu$ as margins. To this end, we write $(X^i_t)_{0\leq t\leq T}$ for the flow generated by \eqref{ODE} with $b=b^i$ and $X^i_0=x,\, i=1,2$. For $p>1$, let $q=p/(p-1)$ be its conjugate number. Then
  \begin{eqnarray*}
  \frac{\d}{\d t}|X_t^1-X_t^2|^p&=&p|X_t^1-X_t^2|^{p-2}\langle X_t^1-X_t^2, b^1(X_t^1)-b^2(X_t^2)\rangle\\
  &\leq& p |X_t^1-X_t^2|^{p-1}|b^1(X_t^1)-b^1(X_t^2)| + p|X_t^1-X_t^2|^{p-1}|b^1(X_t^2)-b^2(X_t^2)|.
  \end{eqnarray*}
Young's inequality leads to
  \begin{eqnarray*}
  \frac{\d}{\d t}|X_t^1-X_t^2|^p &\leq& p L |X_t^1-X_t^2|^p + p\bigg(\frac{|b^1(X_t^2)-b^2(X_t^2)|^p}{p} +\frac{|X_t^1-X_t^2|^p}{q} \bigg)\\
  &\leq& (p L+p-1)|X_t^1-X_t^2|^p + |b^1(X_t^2)-b^2(X_t^2)|^p.
  \end{eqnarray*}
Thus, by Gronwall's lemma, we have
  \begin{eqnarray*}
  |X_t^1-X_t^2|^p\leq \int_0^t e^{(p L+p-1)(t-s)}|b^1(X_s^2)-b^2(X_s^2)|^p\,\d s.
  \end{eqnarray*}
By the definition of the Wasserstein distance, we conclude that
  \begin{eqnarray*}
  W_p(\mu_t^1,\mu_t^2)^p&\leq& \int_{\R^d}|X_t^1-X_t^2|^p\,\d\nu\\
  &\leq&\int_0^t e^{(p L+p-1)(t-s)} \bigg(\int_{\R^d}|b^1(X_s^2)-b^2(X_s^2)|^p\d\nu \bigg)\,\d s\\
  &\leq& e^{(p L+p-1)T} \int_0^t\int_{\R^d}|b^1(x)-b^2(x)|^p\,\d\mu_s^2(x)\d s,
  \end{eqnarray*}
where we used the fact that $(X_t^2)_\#\nu=\mu_t^2$ in the last inequality. From this estimate, we immediately get the uniqueness of solutions to the continuity equation when the vector field is Lipschitz continuous. Similar arguments give rise to the quantitative estimates for the solutions of the Fokker--Planck equation \eqref{FPE}.

Beyond the smooth case, for example, if one only has $b^i\in W^{1,p}(\R^d, \R^d)$, $i=1,2$, then the above arguments no longer work. Nevertheless, using the pointwise characterization of Sobolev functions (cf. \eqref{2-lem-maximal-funct.1} below), Crippa and de Lellis \cite{CrippaDeLellis}  obtained some estimates on the Lagrangian paths of the ODE \eqref{ODE}. For example, they estimated the following quantity
  $$\int_{\R^d} \log\bigg(\frac{|X^1_t(x)-X^2_t(x)|}{\delta} +1\bigg)\,\d x$$
in terms of $\delta$ and the $L^p$-norms of $\nabla b^1,\, b^1- b^2$. Such estimates enable them to give alternative proofs to many of the results in the DiPerna--Lions theory. Motivated by this paper and borrowing some ideas from the theory of optimal transportation, Seis \cite{Seis16a} recently established quantitative stability estimates for solutions of the continuity equation in terms of the Kantorovich--Rubinstein distance. To introduce his result, we need some notations. Fix $\delta>0$. The Kantorovich--Rubinstein distance (see \cite[Chapter 6]{Villani2009} for a discussion on the evolution of the name of this distance) with logrithmic cost function is defined as
  $$\D_\delta(\mu^1,\mu^2)=\inf_{\pi\in \mathcal C(\mu^1,\mu^2)} \int_{\R^d \times\R^d} \log\bigg(\frac{|x-y|}{\delta} +1\bigg) \d\pi(x,y).$$
Such a quantity is finite if
  \begin{equation}\label{prob-space}
  \mu^i\in \mathcal P_{\rm log}(\R^d)=\bigg\{\nu\in \mathcal P(\R^d): \int_{\R^d} \log(1+|x|)\,\d\nu(x)<\infty\bigg\},\quad  i=1,2.
  \end{equation}
As mentioned in \cite[Section 3.1]{Seis16b}, $\D_\delta$ metrizes the weak convergence of probability measures; see also \cite[Theorem 7.12]{Villani}. Seis proved the following estimate: for each $i=1,2$, if the vector field $b^i \in L^1((0,T), W^{1,p}(\R^d, \R^d))$ and $\mu^i_t(\d x)=u^i_t(x)\, \d x$ is a solution to \eqref{CE} such that $u^i\in L^\infty((0,T), L^1\cap L^q(\R^d))$, where $p>1$ and $1/p+1/q=1$, then
  \begin{equation}\label{Seis-estimate}
  \D_\delta(\mu^1_t,\mu^2_t)\leq C_1 + \frac{C_2}{\delta}\|b^1-b^2\|_{L^1(L^p)},
  \end{equation}
where $C_1$ and $C_2$ are constants depending on the norms of the vector fields $b^i$ and the solutions $u^i$. The proof in \cite{Seis16a} is based on the dual formulation of the Kantorovich--Rubinstein distance. We notice that Seis's approach works also for continuity equations with source (cf. \cite[(13)]{Seis16b}), which is applied in \cite[Section 3.3]{Seis16b} to derive results on the zero-diffusivity limit.

On the other hand, the superposition principle (see e.g. \cite[Theorem 3.2]{Ambrosio08}) asserts that, under quite general conditions, any solution to the continuity equation \eqref{CE} is the marginal distribution of a measure $\eta$ on $C(0,T;\R^d)$ supported on integral curves of the time-dependent vector fields $(b_t)_{t\in[0,T]}$. Such a result has also been generalized to the case of Fokker--Planck equations \eqref{FPE}; see \cite[Theorem 2.6]{Figalli} for the case of uniformly bounded coefficients and \cite[Theorem 2.5]{Trevisan} for the case when the coefficients satisfy  \eqref{integrability}. Based on the superposition principle in \cite{Figalli}, R\"ockner and Zhang \cite{RocknerZhang10} proposed a method for proving the uniqueness of solutions to the Fokker--Planck equation with bounded coefficients; see also \cite{Luo14} and the proof of \cite[Theorem 1.3]{Zhang16-2}. A close look at the proof reveals that this method can also yield quantitative stability estimates to the solutions when the coefficients are in the Sobolev space $W^{1,p}$ with $p\geq 1$. We emphasize that our method works for Fokker--Plank equations \eqref{FPE} with degenerate diffusion coefficients. In the non-degenerate case, Bogachev et al. \cite{Bogachev16} recently  established  quantitative estimates on the relative entropy and total variation distance of solutions to  \eqref{FPE} with a different approach, and Manita \cite{Manita} obtained upper bounds for Kantorovich functionals with bounded cost functions between solutions to \eqref{FPE}  with the same diffusion but different dissipative drifts. We would like to mention that our approach may be applied to establish quantitative stability estimates for non-local Fokker--Planck equations, by using the superposition principle recently proved in \cite{Fournier}.

This paper is organized as follows. We state our main results in Section 2. The results in the first subsection belong to the framework of the DiPerna--Lions theory, which deal with the degenerate Fokker--Plank equations \eqref{FPE} with weakly differentiable coefficients, while those in Subsection 2.2 are concerned with non-degenerate equations with an identity diffusion part and a singular drift satisfying an integrability condition. In Section 3, we make the necessary preparations for proving the main results. In particular, we recall the basic notions of solutions to stochastic differential equations and their relationship. We also state Trevisan's superposition principle which generalizes the previous result of Figalli \cite{Figalli} and serves as an important basis of the current work. The proofs of the main results and their applications stated in Subsections 2.1 and 2.2 are given in  Sections 4 and 5, respectively.

\section{Main results and applications}

This section consists of two subsections. In the first one, we state our quantitative stability estimates on Fokker--Planck equations with weakly differentiable coefficients, as well as coefficients satisfying the mixed Osgood and Sobolev condition. In the second part, we consider non-degenerate equations with singular drift satisfying  an integrability  condition.

Fix $\delta>0$. Since we are going to apply It\^o's formula, we shall use the function $s\mapsto \log(\frac{s^2}{\delta^2} +1)$ and consider
\begin{equation}\label{KR-distance}\tilde\D_\delta(\mu^1,\mu^2)=\inf_{\pi\in \mathcal C(\mu^1,\mu^2)} \int_{\R^d \times\R^d} \log\bigg(\frac{|x-y|^2}{\delta^2} +1\bigg) \d\pi(x,y).
\end{equation}
According to the elementary inequality
  \begin{equation}\label{2.1}
  \log(1+s^2)\leq \log(1+2s+s^2)= 2\log(1+s),\quad s\geq 0,
  \end{equation}
the quantity $\tilde\D_\delta(\mu^1,\mu^2)$ is finite if $\mu^i\in \mathcal P_{\rm log}(\R^d)\, (i=1,2)$, where $\mathcal P_{\rm log}(\R^d)$ is defined in \eqref{prob-space}. Note that $\log\big(\frac{|x-y|^2}{\delta^2} +1\big)$ is not a metric on $\R^d$. Since this function is nonnegative and continuous, it is well known that there exists an optimal $\pi_\delta \in \mathcal{C}(\mu^1,\mu^2)$ which achieves the infimum in the definition of $\tilde\D_\delta(\mu^1,\mu^2)$.

\begin{remark}\label{1-rem}
The two quantities $\D_\delta(\mu^1,\mu^2)$ and $\tilde\D_\delta(\mu^1,\mu^2)$ have the following  relations:
  \begin{equation}\label{1-rem.1}
  \tilde\D_\delta(\mu^1,\mu^2)\leq 2 \D_\delta(\mu^1,\mu^2)\quad \mbox{and}\quad \D_\delta(\mu^1,\mu^2) \leq \bigg(\frac{\tilde\D_\delta(\mu^1,\mu^2)}{\log 2}\bigg)^{1/2} +\tilde\D_\delta(\mu^1,\mu^2).
  \end{equation}
The first inequality follows immediately from \eqref{2.1}. As for the second one, we take $\pi_\delta \in \mathcal C(\mu^1,\mu^2)$ such that
  \begin{equation}\label{1-rem.2}
  \tilde\D_\delta(\mu^1,\mu^2)=\int_{\R^d\times \R^d} \log\bigg(\frac{|x-y|^2}{\delta^2}+1\bigg) \d\pi_\delta(x,y).
  \end{equation}
We have
  \begin{equation}\label{1-rem.3}
  \aligned \D_\delta(\mu^1,\mu^2)&\leq \int_{\R^d\times\R^d} \log\bigg(\frac{|x-y|}{\delta} +1\bigg) \d\pi_\delta(x,y)\\
  &= \bigg[\int_{\{|x-y|\leq \delta\}} + \int_{\{|x-y|> \delta\}}\bigg] \log\bigg(\frac{|x-y|}{\delta} +1\bigg) \d\pi_\delta(x,y)\\
  &=: J_1+J_2.
  \endaligned
  \end{equation}
By Cauchy's inequality and using the simple inequality $\log(1+s)\geq (\log 2) s$ for $s\in [0,1]$, we obtain
  $$\aligned J_1&\leq \int_{\{|x-y|\leq \delta\}} \frac{|x-y|}{\delta}\, \d\pi_\delta(x,y)\leq \bigg[\int_{\{|x-y|\leq \delta\}} \frac{|x-y|^2}{\delta^2}\, \d\pi_\delta(x,y)\bigg]^{1/2}\\
  & \leq \bigg[\frac 1{\log 2}\int_{\{|x-y|\leq \delta\}} \log\bigg(\frac{|x-y|^2}{\delta^2}+1\bigg) \d\pi_\delta(x,y)\bigg]^{1/2}\leq \bigg(\frac{\tilde\D_\delta(\mu^1,\mu^2)}{\log 2}\bigg)^{1/2},
  \endaligned$$
where the last inequality follows from \eqref{1-rem.2}. Next,
  $$J_2\leq \int_{\{|x-y|> \delta\}} \log\bigg(\frac{|x-y|^2}{\delta^2} +1\bigg) \d\pi_\delta(x,y)\leq \tilde\D_\delta(\mu^1,\mu^2).$$
Substituting these estimates into \eqref{1-rem.3} leads to \eqref{1-rem.1}.
\end{remark}

In the following, we write $\|\cdot\|_{L^r(L^s)}$ for the norm in the function space $L^r(0,T; L^s(\R^d)),\, r, s\in[1,\infty]$. Though the image space is not explicitly given, there will be no confusion according to the context.


\subsection{Degenerate equations with weakly differentiable coefficients}

In this part we work in the framework of the DiPerna--Lions theory and consider Fokker--Planck equations \eqref{FPE} with weakly differentiable coefficients.  The first main result is

\begin{theorem}\label{2-thm-1}
Let $p> 1$ and $q$ be its conjugate number. For each $i\in\{1,2\}$, assume that $\sigma^i\in L^2(0,T; W^{1,2p}(\R^d, \mathcal{M}_{d,m}))$ and $b^i\in L^1(0,T; W^{1,p}(\R^d,\R^d))$, and $u^i\in L^\infty(0,T; L^1\cap L^q(\R^d))$ is the solution to the corresponding Fokker--Planck equation \eqref{FPE} with $a=\sigma^i\,(\sigma^i)^\ast$ and $b=b^i$. Let $\mu^i_t= u^i_t\, \d x$ and assume that $\mu^i_0\in \mathcal P_{\rm log}(\R^d),\, i=1,2$. Then for all $t\in [0,T]$,
  \begin{equation}\label{2-thm-1-0}
  \begin{split}
  \tilde\D_\delta(\mu^1_t, \mu^2_t)&\leq \tilde\D_\delta(\mu^1_0, \mu^2_0) +2\|u^2\|_{L^\infty(L^q)} \bigg(\frac1\delta \|b^1-b^2\|_{L^1(L^p)} + \frac1{\delta^2} \|\sigma^1-\sigma^2\|_{L^2(L^{2p})}^2\bigg)\\
  &\hskip13pt + C_{d,p} \bigg(\sum_{i=1}^2\|u^i\|_{L^\infty(L^q)}\bigg) \Big(\big\| \nabla b^1\big\|_{L^1(L^p)} +\big\|\nabla \sigma^1\big\|_{L^2(L^{2p})}^2 \Big),
  \end{split}
  \end{equation}
where $C_{d,p}$ is a positive constant depending only on $d$ and $p$.
\end{theorem}

Under our assumptions, it is clear that the solution $(\mu^i_t)_{0\leq t\leq T}$ satisfies \eqref{integrability}. Thus, by Lemma \ref{2-lem-growth} below, we have $\mu^i_t\in \mathcal P_{\rm log}(\R^d)$ for any $t\in [0,T]$ and $i=1,2$, which implies that the quantity $\tilde\D_\delta(\mu^1_t, \mu^2_t)$ is finite.

Here is a comment on the above stability result.  The estimate \eqref{2-thm-1-0} is of little use for a fixed value of $\delta>0$; instead, it should be understood in a dynamical way. More precisely, suppose that $b^2$ and $\sigma^2$ are replaced by two sequences $\{b^n\}_{n\geq2}$ and $\{\sigma^n\}_{n\geq2}$, which converge in $L^1(0,T; L^p(\R^d))$ to $b^1$ and in $L^2(0,T; L^{2p}(\R^d))$ to $\sigma^1$, respectively. Let $\mu^n_t= u^n_t(x)\,\d x$ be the solution of \eqref{FPE}  corresponding to the coefficients $b^n$ and $a^n=\sigma^n (\sigma^n)^\ast$ and $\mu^n_0=\mu^1_0$. Assume that the sequence $\{u^n\}_{n\geq 1}$ is bounded in $L^\infty(0,T; L^1\cap L^q(\R^d))$. If we take
  $$\delta=\delta_n= \|b^1-b^n\|_{L^1(L^p)}+ \|\sigma^1-\sigma^n\|_{L^2(L^{2p})},$$
then \eqref{2-thm-1-0} implies that $\tilde\D_{\delta_n}(\mu^1_t,\mu^n_t)\leq \tilde C<\infty$. From this we conclude that,  as $n\to \infty$, $\mu^n_t$ tends to $\mu^1_t$ at the speed of  $\delta_n$ (see Proposition \ref{prop-zero-diffusivity} for a related result).

\begin{remark}\label{2-rem}
(1) The existence of solutions to the Fokker--Planck equation \eqref{FPE} in the space $L^\infty(0,T; L^q(\R^d))$ follows from standard 
a priori   
estimates (see e.g. \cite[Section 5.2]{LebrisLions08}), provided that $(\div(b^\sigma))^- \in L^1(0,T; L^\infty(\R^d))$, where $b^\sigma=b-\frac12 \div(\sigma\sigma^\ast)$.

(2) The second term on the right hand side of \eqref{2-thm-1-0} can be replaced by
  $$2\bigg(\frac1\delta \int_0^T\!\! \int_{\R^d} |b^1_t-b^2_t|\,\d\mu^2_t\d t + \frac1{\delta^2} \int_0^T\!\! \int_{\R^d} \|\sigma^1_t-\sigma^2_t\|^2\,\d\mu^2_t\d t\bigg),$$
where the density of $\mu^2_t$ does not appear. But for the last term, as the maximal inequality \eqref{2-lem-maximal-funct.2} below holds only for the Lebesgue measure, we have to assume the existence of density and use the H\"older inequality to separate it from the other terms; see the proof in Section 4.1 for details.

(3) Similar to the above remark, the assumptions that $\sigma^i\in L^2(0,T; L^{2p}(\R^d, \mathcal M_{d,m}))$ and $b^i\in L^1(0,T; L^p(\R^d, \R^d))$ can be replaced by
  $$\int_0^T\!\! \int_{\R^d} (\|\sigma^i_t\|^2 +|b^i_t|)\,\d\mu^i_t\d t<+\infty,$$
but we do require that $\nabla \sigma^i_{\alpha\beta}\in L^2(0,T; L^{2p}(\R^d,\R^d))$ and $\nabla b^i_\alpha\in L^2(0,T; L^p(\R^d,\R^d))$, where $\alpha=1,\cdots,d$ and $\beta=1,\cdots,m$.
\end{remark}

\begin{remark}\label{2-rem-2}
In \cite{LebrisLions08}, the authors considered the following Fokker--Planck equation of divergence form:
  $$\partial_t u_t+\div(u_t b) -\frac12\sum_{ij} \partial_i((\sigma\sigma^\ast)_{ij}\partial_j u_t)=0.$$
It is easy to see that the above equation is equivalent to
  $$\partial_t u_t+\div(u_t b_\sigma) -\frac12\sum_{ij} \partial_{ij}((\sigma\sigma^\ast)_{ij} u_t)=0,$$
where $b_\sigma= b+\frac12\div(\sigma\sigma^\ast)$. Therefore, we can apply Theorem \ref{2-thm-1} to get quantitative stability estimate for solutions of Fokker--Planck equation of divergence form. However, if one attempt to transform the backward Kolgomorov equation
  $$\partial_t u_t-b\cdot\nabla u_t -\frac12\sum_{ij}(\sigma\sigma^\ast)_{ij} \partial_{ij} u_t=0$$
to the Fokker--Planck equation \eqref{FPE}, then there is an extra zeroth order term, which prevents the application of our results.
\end{remark}

Since the inequality \eqref{2-lem-maximal-funct.2} below for maximal functions is not valid when $p=1$, we shall adapt a technique from \cite[Theorem 4.1]{Jabin} to show the following result.

\begin{theorem}\label{2-thm-2}
For each $i\in\{1,2\}$, suppose that  $\sigma^i\in L^2(0,T; W^{1,2}(\R^d,\mathcal{M}_{d,m}))$ and $b^i\in L^1(0,T; W^{1,1}(\R^d,\R^d))$, and $u^i\in L^\infty(0,T; L^1\cap L^\infty(\R^d))$ is the solution to the corresponding Fokker--Planck equation \eqref{FPE} with $a=\sigma^i\,(\sigma^i)^\ast$ and $b=b^i$. Let $\mu^i_t= u^i_t\, \d x$ and assume that $\mu^i_0\in \mathcal P_{\rm log}(\R^d),\, i=1,2$. Then for all $t\in [0,T]$,
  \begin{equation*}
  \begin{split}
  \tilde\D_\delta(\mu^1_t, \mu^2_t)&\leq \tilde\D_\delta(\mu^1_0, \mu^2_0) +2\|u^2\|_{L^\infty(L^\infty)} \bigg(\frac1\delta \|b^1-b^2\|_{L^1(L^1)} + \frac1{\delta^2} \|\sigma^1-\sigma^2\|_{L^2(L^2)}^2\bigg)\\
  &\hskip13pt + C_{d,T} \bigg(1+\sum_{i=1}^2\|u^i\|_{L^\infty(L^\infty)}\bigg) \Big[\phi(\delta) \big(1+\|G(|\nabla b^1|)\|_{L^1(L^1)}\big) + \big\| \nabla \sigma^1\big\|_{L^2(L^2)}^2 \Big],
  \end{split}
  \end{equation*}
where $\frac{\phi(\delta)}{|\log \delta|}\to 0$ as $\delta\downarrow 0$, $G:\R_+\to \R_+$ is some convex function such that $G(|\nabla b^1|)\in L^1(0,T; L^1(\R^d))$, and $C_{d,T}$ is a positive constant depending only on $d$ and $T$.
\end{theorem}

Now we consider the Fokker--Planck equation \eqref{FPE} whose coefficients $\sigma$ and $b$ satisfy the following mixed Osgood and Sobolev type condition (see \cite[Example 2.4]{LiLuo} for an example of such a function):

\begin{enumerate}
\item[$(\mathbf{H})$] There exist a nonnegative function $g\in L^1(0,T; L^1(\R^d))$, such that for all $t\in [0,T]$, for a.e. $x,y\in \R^d$, one has
  \begin{equation}\label{Osgood-Sobolev.1}
  |\<x-y,b_t(x)-b_t(y)\>| + \|\sigma_t(x)-\sigma_t(y)\|^2 \leq \big(g_t(x) + g_t(y)\big) \rho(|x-y|^2),
  \end{equation}
where $\rho\in C^1(\R_+,\R_+)$ is a nondecreasing function satisfying $\rho(0)=0$ and $\int_{0+}\frac{\d s}{\rho(s)}=\infty$.
\end{enumerate}

Without loss of generality, we shall assume $\rho(s)\geq s$ for all $s\geq 0$. In the current case, instead of using the auxiliary function $\log(1+s^2/\delta^2)$, we define
  \begin{equation}\label{psi-delta}\psi_\delta(s)=\int_0^s \frac{\d r}{\rho(r) +\delta^2},\quad s\geq0.  \end{equation}
Note that $\lim_{\delta \downarrow 0} \psi_\delta(s^2)=\infty$ and $\psi_\delta(s)= \log(1+s^2/\delta^2)$ if $\rho(s)=s$ for all $s\geq 0$. Moreover,
  \begin{equation}\label{auxi-funct}
  \psi'_\delta(s)=\frac1{\rho(s) +\delta^2}>0,\quad \psi''_\delta(s)=-\frac{\rho'(s)}{(\rho(s) +\delta^2)^2}\leq 0.
  \end{equation}
This property shows that $\psi_\delta$ is a concave function for any $\delta>0$. For two probability measures $\mu$ and $\nu$ on $\R^d$, we define $\D_{\psi_\delta}(\mu, \nu)$ as above by replacing $\log(1+s^2/\delta^2)$ in \eqref{KR-distance} with $\psi_\delta(s^2)$.

\begin{theorem}\label{4-thm}
For $i=1,2$, assume that $\sigma^i\in L^2(0,T; L^2(\R^d,\mathcal{M}_{d,m}))$, $b^i\in L^1(0,T; L^1(\R^d,\R^d))$, and the hypothesis $(\mathbf{H})$ holds for $\sigma^1$ and $b^1$. Let $\mu^i_t =u^i_t\,\d x$ be the  solution to the Fokker--Planck equation \eqref{FPE} with $a=\sigma^i\,(\sigma^i)^\ast$ and $b=b^i$ such that $u^i\in L^\infty (0,T; L^1\cap L^\infty(\R^d) ),\, i=1, 2$. Then for all $t\in [0,T]$,
  \begin{eqnarray*}
  \D_{\psi_\delta}(\mu^1_t,\mu^2_t) &\leq& \D_{\psi_\delta}(\mu^1_0,\mu^2_0) + 8\|g\|_{L^1(L^1)} \sum_{i=1}^2 \big\|u^i \big\|_{L^\infty(L^\infty)}\\
  &&+ 2\|u^2\|_{L^\infty(L^\infty)} \Big(\frac1\delta \|b^1-b^2\|_{L^1(L^1)}+ \frac1{\delta^2} \big\|\sigma^1 -\sigma^2\big\|^2_{L^2( L^2)}\Big).
  \end{eqnarray*}
\end{theorem}

We provide some applications of our results to conclude this subsection. As a direct consequence, we immediately get the uniqueness of solutions to the Fokker--Planck equation \eqref{FPE}.

\begin{corollary}\label{cor-uniqueness}
Assume that  $\sigma\in L^2(0,T; W^{1,2p}(\R^d,\mathcal{M}_{d,m}))$ and $b\in L^1(0,T; W^{1,p}(\R^d,\R^d))$ with $p\geq1$. Then the Fokker--Planck equation \eqref{FPE} has at most one solution in the class $L^\infty(0,T; L^1\cap L^q(\R^d))$, where $\frac1p+\frac1q =1$.
\end{corollary}

We can also deduce the uniqueness of the Fokker--Planck equation \eqref{FPE} in the framework of Theorem \ref{4-thm} by the same method, which we omit here.

Next we consider the
zero diffusivity limit which is inspired by \cite[Section 3.3]{Seis16b}. Let $\kappa>0$. Unlike the equation (14) in \cite{Seis16b} which is non-degenerate with constant diffusion coefficient, we consider the equation
  \begin{equation}\label{zero-diffusivity}
  \partial_t\rho^\kappa_t+\div(\rho^\kappa_t b)= \frac\kappa2 \sum_{ij}\partial_{ij} (\rho^\kappa_t a_{ij} ),\quad \rho^\kappa_0=\bar\rho,
  \end{equation}
where $a=\sigma\sigma^\ast$ and $\sigma\in C_b^2(\R^d,\mathcal M_{d,m})$ is a time-independent matrix-valued function. Assume that $\bar\rho \in L^1\cap L^\infty(\R^d, \R_+)$ and $(\div(b))^- \in L^1(0,T; L^\infty(\R^d))$, then it is easy to show that (see \cite[Remark 4]{LebrisLions08} or \cite[(5)]{Seis16b})
  \begin{equation}\label{zero-diffusivity.1}
  \sup_{t\in [0,T]} \|\rho^\kappa_t\|_{L^q}\leq \|\bar\rho\|_{L^q} \exp\bigg[\Big(1-\frac1q \Big)\int_0^T \|(\div(b_t))^-\|_{L^\infty} \,\d t\bigg].
  \end{equation}
Note that this estimate is independent of $\kappa$, and also holds in the case when $\kappa=0$, i.e., the solution of the continuity equation \eqref{CE}. We present the following vanishing diffusivity limit which reveals that, as $\kappa\to 0$, $\rho^\kappa_t$ converges weakly to $\rho^0_t$ as fast as $\sqrt \kappa$.

\begin{proposition}\label{prop-zero-diffusivity}
Let $p>1$ and $\rho^\kappa_t$ (resp. $\rho_t^0$) be the solution to equation \eqref{zero-diffusivity} (resp. equation \eqref{CE}) with initial value $\bar\rho\in L^1\cap L^\infty(\R^d, \R_+)$. Assume that $\sigma\in C_b^2(\R^d,\mathcal M_{d,m})\cap L^{2p}(\R^d,\mathcal M_{d,m})$ and $b\in L^1(0,T; W^{1,p}(\R^d, \R^d))$ with $[\div(b)]^- \in L^1(0,T; L^\infty(\R^d))$. Then for every sufficiently small $\kappa>0$, we have
  $$\sup_{t\in[0,T]} \tilde\D_{\sqrt \kappa}(\rho^\kappa_t, \rho^0_t)\leq 2 C_{q,T} \big(C_{d,p}\|\nabla b\|_{L^1(L^p)} + T\|\sigma\|_{L^{2p}}^2\big),$$
where $C_{d,p}$ is a positive constant depending only on $d$ and $p$, $C_{q,T}$ denotes the right hand side of \eqref{zero-diffusivity.1},  and $q=p/(p-1)$.
\end{proposition}

Finally, we provide a variant of Theorem \ref{2-thm-1} which will be used below. The basic idea of proof is to apply H\"older's inequality to the diffusion and the drift terms with different exponents.

\begin{theorem}\label{thm-variant}
Let $p_1> 1,\ p_2>1$ and $q=\frac{p_1}{p_1 -1} \vee \frac{p_2}{p_2 -1}$. For each $i\in\{1,2\}$, assume that $\sigma^i\in L^2(0,T; W^{1,2p_1}(\R^d,\mathcal{M}_{d,m}))$ and $b^i\in L^1(0,T; W^{1,p_2}(\R^d,\R^d))$, and $u^i\in L^\infty(0,T; L^1\cap L^q(\R^d))$ is the solution to the corresponding Fokker--Planck equation \eqref{FPE} with $a=\sigma^i\,(\sigma^i)^\ast$ and $b=b^i$. Let $\mu^i_t= u^i_t\, \d x$ and assume that $\mu^i_0\in \mathcal P_{\rm log}(\R^d),\, i=1,2$. Then for all $t\in [0,T]$,
  \begin{equation*}
  \begin{split}
  \tilde\D_\delta(\mu^1_t, \mu^2_t)&\leq \tilde\D_\delta(\mu^1_0, \mu^2_0) + C_1 \bigg(\frac1{\delta^2}\|\sigma^1-\sigma^2\|_{L^2(L^{2p_1})}^2 + \big\| \nabla \sigma^1\big\|_{L^2(L^{2p_1})}^2 \bigg)\\
  &\hskip13pt + C_2 \bigg(\frac1{\delta}\|b^1-b^2\|_{L^1(L^{p_2})} + \big\| \nabla b^1\big\|_{L^1(L^{p_2})} \bigg),
  \end{split}
  \end{equation*}
where the positive constant $C_k$ depends on $\sum_{i=1}^2 \|u^i\|_{L^\infty(L^{p'_k})}$ with $p'_k=\frac{p_k}{p_k -1},\ k=1,2$.
\end{theorem}

\subsection{Non-degenerate equations with singular drifts}

So far we are mainly concerned with Fokker--Planck equations with degenerate diffusion coefficient; for this reason the drift coefficient is usually required to fulfill some weak differentiability. Next we consider the  non-degenerate equation
  \begin{equation}\label{FPE-LPS}
  \partial_t \mu_t +\div(\mu b_t) -\frac12 \Delta \mu_t=0,\quad \mu|_{t=0} =\mu_0,
  \end{equation}
where the drift satisfies only the Ladyzhenskaya--Prodi--Serrin condition, i.e.,
  \begin{equation}\label{LPS}
  b\in L^q(0,T; L^p(\R^d,\R^d))\quad \mbox{with  }p>2,\ q>2\mbox{ such that } \frac dp +\frac2q <1.
  \end{equation}
Recall that Bogachev et al. \cite{Bogachev16} proved quantitative stability estimates to non-degenerate Fokker--Planck equations, but they assume
the drift $b$ to be locally bounded. Under the condition \eqref{LPS}, the It\^o SDE
  \begin{equation}\label{Ito-SDE-1}
   \d X_t= \d B_t+ b_t(X_t) \,\d t,\quad X_0=x,
  \end{equation}
has been studied intensively since the seminal  paper of Krylov and R\"ockner \cite{Krylov}; see also \cite{Zhang11,Flandoli11,Flandoli,LiLuoWany,Zhang16}. It was shown that \eqref{Ito-SDE-1} determines a stochastic flow of H\"older continuous homeomorphisms on  $\R^d$. The basic tool of proof is Zvonkin's transformation which transforms the equation \eqref{Ito-SDE-1} into a new one with regular coefficients. For later use, here we briefly recall the main steps of this method (see \cite[Section 3.1]{Flandoli}).

For $\lambda>0$, the vector-valued backward parabolic equation
  \begin{equation}\label{Zvonkin}
  \partial_t \phi_t+\frac12 \Delta \phi_t + b_t\cdot\nabla \phi_t-\lambda \phi_t = -b_t,\quad \phi_T(x)=0\in \R^d
  \end{equation}
has a unique solution $\phi\in H^q_{2,p}(T) := L^q(0,T; W^{2,p}(\R^d, \R^d)) \cap W^{1,q}(0,T; L^p(\R^d, \R^d))$ such that
  \begin{equation}\label{regularity}
  \|\partial_t \phi\|_{L^q(L^p)}+ \|\phi\|_{L^q(W^{2,p})}\leq C\|b\|_{L^q(L^p)},
  \end{equation}
where $C$ is a positive constant depending on $d,p,q,T,\lambda$ and $\|b\|_{L^q(L^p)}$; moreover, when $\lambda$ is big enough, we have $\sup_{t\in[0,T]} \|\nabla\phi_t\|_{L^\infty}\leq 1/2$. Define
  \begin{equation}\label{Zvonkin.1}
  \psi_t(x)= x+ \phi_t(x),\quad x\in \R^d,
  \end{equation}
then $\psi_t:\R^d\to \R^d$ is a diffeomorphism with bounded first derivatives, uniformly in $t\in[0,T]$, and the same is true for the inverse mappings $\psi_t^{-1}:\R^d\to \R^d$. Now, let $Y_t= \psi_t(X_t),\, 0\leq t\leq T$, which solves SDE
  \begin{equation}\label{Ito-SDE-2}
   \d Y_t= \tilde\sigma_t(Y_t)\,\d B_t+ \tilde b_t(Y_t) \,\d t,
  \end{equation}
where (${\rm Id}$ is the $d\times d$ identity matrix)
  \begin{equation}\label{Ito-SDE-2.1}
  \tilde\sigma_t(y)={\rm Id} + (\nabla \phi_t)\circ \psi_t^{-1}(y),\quad \tilde b_t(y)=\lambda \phi_t\circ \psi_t^{-1}(y),\quad y\in\R^d.
  \end{equation}
The coefficients $\tilde\sigma$ and $\tilde b$ of \eqref{Ito-SDE-2} are much more regular than $b$, which makes it possible to establish  some key estimates on the solution $Y_t$, and then transfer them back to the solution $X_t$ of \eqref{Ito-SDE-1}. We shall use this idea to prove the quantitative stability estimates for the solutions of \eqref{FPE-LPS}.

\begin{theorem}\label{thm-LPS}
Let $i\in\{1,2\}$. Suppose that $b^i$ satisfies \eqref{LPS} and $\mu^i$ is the solution of \eqref{FPE-LPS} with $b=b^i$. Assume that $\mu^i_t = u^i_t\,\d x$ with $u^i\in L^\infty(0,T; L^1\cap L^{p/(p-2)}(\R^d))$. Then for all $t\in [0,T]$,
  $$\aligned
  \tilde\D_\delta(\mu^1_t, \mu^2_t)&\leq \tilde\D_{\delta/9}(\mu^1_0, \mu^2_0) +C_1 \bigg(\frac1{\delta^2}\|b^1- b^2\|_{L^q(L^p)}^2 + \| b^1\|_{L^q(L^p)}^2 \bigg)\\
  &\hskip13pt + C_2 \bigg(\frac1{\delta}\|b^1- b^2\|_{L^q(L^p)} + \| b^1\|_{L^q(L^p)} \bigg), \endaligned$$
where $C_1$ and $C_2$ are some positive constants.
\end{theorem}

\begin{remark}
(1) Formally, the equation \eqref{FPE-LPS} is equivalent to
  $$\partial_t \mu_t =\frac12 \Delta \mu_t- b_t\cdot\nabla \mu_t -\div(b_t)\mu_t,\quad \mu|_{t=0} =\mu_0.$$
If the distributional divergence $\div(b)$ of $b$ exists and belongs to, say, $L^\infty(0,T; L^\infty(\R^d))$, then one can establish an estimate of type  \eqref{regularity}, which gives us a stability estimate in the space $H^q_{2,p}(T)$. However, this method requires some regularity conditions on $\div(b)$.

(2) Using the regularity results in \cite[Section 3]{Zhang16} on the solutions of parabolic equations, we can also consider the equation \eqref{FPE-LPS} with general non-degenerate diffusion coefficients, but we have to restrict to a small time interval (see \cite[Lemma 6.1]{Zhang16}). In order to obtain estimates on any given time interval $[0,T]$, we should add the term $-\lambda u$ to the equation considered in \cite[Section 3]{Zhang16}, and repeat the arguments therein to establish regularity results on solutions for $\lambda$ big enough. In this work we do not want to enter into such details.
\end{remark}

Finally, inspired by \cite[Theorem 1.1(E)]{Zhang16}, we consider the classical Kantorovich--Wasserstein distance $W_2$ and prove

\begin{theorem}\label{thm-LPS-Wass}
Let $\alpha\in (2,p\wedge q)$.  For each $i=1,2$, suppose that $b^i$ satisfies \eqref{LPS} and  $\mu^i_t= u^i_t(x)\,\d x$ is the solution to \eqref{FPE-LPS} satisfying $u^i\in L^\infty \big(0,T; L^1\cap L^{p/(p-\alpha)}(\R^d)\big)$ and $\mu^i_0$ has finite moment of order $2\alpha-2$. Then there is a constant $C_\alpha>0$ such that for all $t\in [0,T]$,
\begin{eqnarray}\label{W-estimate}
W_2(\mu^1_t, \mu^2_t)\leq C_\alpha\Big[W_\alpha(\mu^1_0, \mu^2_0) + \|u^2\|_{L^\infty(L^{p/(p-\alpha)})}^{1/\alpha} \|b^1-  b^2\|_{L^q (L^p)} \Big].
\end{eqnarray}
\end{theorem}

Note that, on the right hand side, we use the $\alpha$-Kantorovich--Wasserstein distance $W_\alpha$ with $\alpha$ bigger than 2. It seems that one cannot replace $W_\alpha$ by $W_2$ (cf. \cite[Lemma 4.2]{Zhang13}).

\begin{remark}
We consider the special case that $\mu^1_0=\mu^2_0$.

(1) The assumptions of Theorem \ref{thm-LPS-Wass} are stronger than those in Theorem \ref{thm-LPS}. Using the simple inequalities
  $$\tilde \D_\delta\leq 2\D_\delta\leq \frac2\delta W_1\leq \frac2\delta W_2,$$
we see that \eqref{W-estimate} implies the result of Theorem \ref{thm-LPS}.

(2) Recall that Bogachev et al. \cite[Corollary 1.2]{Bogachev16} proved quantitative stability estimates on the total variation distance between solutions to Fokker--Planck equations. Our estimate \eqref{W-estimate} is similar to theirs when the diffusion coefficient is the identity matrix and $\varphi=1$.
\end{remark}

\section{Preparations}

In this section, we make some preparations which are important for the proofs of the main results. We need the following basic results in harmonic analysis. For $f\in L^1_{loc}(\R^d)$, denote by $M f(x)$ the maximal function of $f$, i.e.,
  $$M f(x):=\sup_{r>0}\frac1{|B_r|}\int_{B_r} |f(x+y)|\,\d y,$$
where $B_r$ is the ball centered at the origin with radius $r$, and $|B_r|$ is its Lebesgue measure.

\begin{lemma}\label{2-lem-maximal-funct}
There is a dimensional constant $C_d>0$ such that, for any $f\in W^{1,1}_{loc}(\R^d)$,
  \begin{equation}\label{2-lem-maximal-funct.1}
  |f(x)-f(y)|\leq C_d|x-y|\big(M|\nabla f|(x) + M|\nabla f|(y)\big),\quad \mbox{for all } x,y\in L(f),
  \end{equation}
where $L(f)\subset \R^d$ is the set of Lebesgue points of $f$. Moreover, for any $p>1$, there is a constant $C_{d,p}>0$ such that
  \begin{equation}\label{2-lem-maximal-funct.2}
  \int_{\R^d} \big(M f(x)\big)^p\,\d x\leq C_{d,p} \int_{\R^d} |f(x)|^p\,\d x, \quad\mbox{for every } f\in L^p(\R^d).
  \end{equation}
\end{lemma}

Note that if $L(f)=\emptyset$ (for instance, when $f$ is continuous), then \eqref{2-lem-maximal-funct.1} holds for all $x,y\in\R^d$ (cf. \cite[Appendix]{FangLuoThalmaier}). Let $\mathbb W_T^d= C([0,T],\R^d)$ be the space of continuous functions from $[0,T]$ to $\R^d$. Let $\F_t^d$ be the canonical filtration generated by the coordinate process $e_s(w)=w_s$, for $w\in \mathbb W_T^d$ and $0\leq s\leq t$, where $e_s: \mathbb W_T^d\to \R^d$ is the evaluation map. The next technical result will play an important role in the proof of the main result.

\begin{lemma}\label{2-lem}
Let $\eta^1$ and $\eta^2$ be two probability measures on the path space $\W^d_T$. Denote by $\mu^i_0=(e_0)_\# \eta^i,\, i=1,2$. Then for any $\pi\in \mathcal C(\mu^1_0,\mu^2_0)$, there exists a probability space  $(\Omega, \mathcal F, \P)$ on which there are defined two stochastic processes $(Y^1_t)_{0\leq t\leq T}$ and  $(Y^2_t)_{0\leq t\leq T}$, such that $\eta^i$ is the distribution of $(Y^i_t)_{0\leq t\leq T}$ on $\W^d_T$, $i=1,2$, and $\pi$ is the joint distribution of  $(Y^1_0, Y^2_0)$.
\end{lemma}

\begin{proof}
For each $i=1,2$, we disintegrate $\eta^i$ with respect to $\mu^i_0$ as follows:
  $$\d\eta^i(w)= \d\eta^i_x(w)\d\mu^i_0(x),$$
where $\eta^i_x$ is the regular conditional probability on $\W^d_T$ which concentrates on the paths starting from $x$ at time 0. We define $\Omega=\W^d_T\times \W^d_T= \W^{2d}_T$, $\mathcal F=\mathcal B(\W^{2d}_T)$ (the family of Borel subsets of $\W^{2d}_T$), and the probability measure
  $$\d\P(w^1,w^2)=\d\eta^1_x(w^1)\d\eta^2_y(w^2)\d\pi(x,y).$$
Now we can set $Y^i_t(w^1,w^2)=w^i_t\ (t\in [0,T], i=1,2)$ to be the coordinate process as usual. To show that $Y^1=(Y^1_t)_{0\leq t\leq T}$ is distributed as $\eta^1$, letting $A\in \mathcal B(\W^d_T)$, we have
  $$\aligned \P(Y^1\in A)&= \P(\{(w^1,w^2): w^1\in A\})= \int_{A\times \W^d_T} \d\P(w^1,w^2)\\
  &= \int_{\R^d\times \R^d} \d\pi(x,y) \int_{A\times \W^d_T} \d\eta^1_x(w^1)\d\eta^2_y(w^2) = \int_{\R^d\times \R^d} \eta^1_x(A)\, \d\pi(x,y)\\
  &= \int_{\R^d} \eta^1_x(A)\, \d\mu^1_0(x) =\eta^1(A). \endaligned$$
In the same way, we see that $Y^2=(Y^2_t)_{0\leq t\leq T}$ has distribution $\eta^2$. Finally, for any two Borel sets $E,F\in \mathcal B(\R^d)$,
  $$\aligned \P((Y^1_0, Y^2_0)\in E\times F)&= \P(\{(w^1,w^2): w^1_0\in E, w^2_0\in F\})\\
  &= \int_{\R^d\times \R^d} \d\pi(x,y) \int_{\{(w^1,w^2):\, w^1_0\in E,\, w^2_0\in F\}} \d\eta^1_x(w^1)\d\eta^2_y(w^2)\\
  &= \int_{\R^d\times \R^d} \eta^1_x(\{w^1: w^1_0\in E\})\,  \eta^2_y(\{w^2: w^2_0\in F\})\, \d\pi(x,y)\\
  &= \pi(E\times F),    \endaligned$$
which implies that the joint distribution of $(Y^1_0, Y^2_0)$ is  $\pi$.
\end{proof}

\subsection{Martingale solution, weak solution and superposition principle}

We give some further preparations which are mainly taken from \cite[Section 2]{Luo14}; see also the beginning parts of \cite[Sections 1 and 2]{RocknerZhang10}. Recall that $\mathcal P(\R^d)$ is the set of probability measures on $(\R^d,\mathcal B(\R^d))$. To fix the notations, we  state in detail the two well known notions of solutions to \eqref{Ito-SDE}.

\begin{definition}[Martingale solution]
Given $\mu_0\in \mathcal P(\R^d)$, a probability measure $P_{\mu_0}$ on $(\mathbb W_T^d,\F_T^d)$ is called a martingale solution to SDE \eqref{Ito-SDE} with initial distribution $\mu_0$ if $(e_0)_\# P_{\mu_0} =\mu_0$, and for any $\varphi\in C^{1,2}([0,T]\times \R^d)$, $\varphi(t,w_t)-\varphi(0,w_0)-\int_0^t (\partial_s \varphi+\L \varphi)(s, w_s)\,\d s$ is an $(\F_t^d)$-martingale under $P_{\mu_0}$, where $\L$ is the time-dependent infinitesimal generator of \eqref{Ito-SDE}.
\end{definition}

\begin{remark}
Let $P_{\mu_0}$ be a martingale solution to \eqref{Ito-SDE}. Define $\mu_t:= (e_t)_\# P_{\mu_0}\in \mathcal P(\R^d),\, t\in [0,T]$. If $(\mu_t)_{0\leq t\leq T}$ satisfies \eqref{integrability}, then it is a weak solution of \eqref{FPE}. It suffices to verify \eqref{FPE-1}. To see this, we first show that $t\mapsto \int_{\R^d}\varphi_t\,\d \mu_t$ is absolutely continuous. Let $n\in\mathbb{N}$ and $\{(s_k,t_k)\}_{1\leq k\leq n}$ be a family of disjoint subintervals of $[0,T]$. By the definition of martingale solution,
  $$\aligned
  \int_{\mathbb W_T^d} \varphi_{t_k}(w_{t_k})\,\d P_{\mu_0} -\int_{\mathbb W_T^d} \varphi_{s_k}(w_{s_k})\,\d P_{\mu_0} =\int_{s_k}^{t_k}\! \int_{\mathbb W_T^d} (\partial_s \varphi+\L \varphi)(s, w_s)\,\d P_{\mu_0}\d s.
  \endaligned$$
Thus,
  $$\aligned
  \sum_{k=1}^n \bigg|\int_{\R^d}\varphi_{t_k}\,\d \mu_{t_k}- \int_{\R^d}\varphi_{s_k}\,\d \mu_{s_k}\bigg| &\leq \sum_{k=1}^n \int_{s_k}^{t_k}\!\int_{\R^d} \big|(\partial_s \varphi+\L \varphi)(s, x)\big|\,\d\mu_s(x)\d s\\
  &\leq \|\varphi\|_{C^{1,2}} \sum_{k=1}^n \int_{s_k}^{t_k}\!\int_{\R^d} \Big(1+\frac12 \|a_s\|+|b_s| \Big)\,\d\mu_s\d s,
  \endaligned$$
where
  $$\|\varphi\|_{C^{1,2}}=\sup_{(t,x)\in [0,T]\times \R^d}(|\varphi(t,x)|+|\partial_t\varphi(t,x)|+|\nabla\varphi(t,x)|+|\nabla^2\varphi(t,x)|).$$
The integrability condition \eqref{integrability} implies the desired result. Now we can compute the time derivative to get
  $$\frac{\d}{\d t} \int_{\R^d}\varphi_t\,\d \mu_t= \int_{\R^d} (\partial_t \varphi+\L \varphi)(t, x)\,\d\mu_t(x),\quad \mbox{for a.e. } t\in [0,T].$$
The equality \eqref{FPE-1} follows by integrating on $[0,T]$.
\end{remark}

\begin{definition}[Weak solution]\label{sect-2-def-2}
Let $\mu_0\in\mathcal P(\R^d)$. The SDE \eqref{Ito-SDE} is said to have a weak solution with initial law $\mu_0$ if there exist a filtered probability space $(\Omega,\mathcal G,(\mathcal G_t)_{0\leq t\leq T},P)$, on which there are defined a $(\mathcal G_t)$-adapted continuous process $X_t$ taking values in $\R^d$ and an $m$-dimensional standard $(\mathcal G_t)$-Brownian motion $W_t$ such that, $X_0$ is distributed as $\mu_0$ and a.s.,
  \begin{equation}\label{sect-2-def-2.1}
  X_t=X_0+\int_0^t \sigma(X_s)\,\d W_s+\int_0^t b(X_s)\,\d s,\quad \forall\, t\in[0,T].
  \end{equation}
We denote this solution by $\big(\Omega,\mathcal G,(\mathcal G_t)_{0\leq t\leq T},P;X,W\big)$.
\end{definition}

Let $\mu_t:= \mbox{law}(X_t),\, 0\leq t\leq T$. If $(\mu_t)_{0\leq t\leq T}$ satisfies the assumption \eqref{integrability}, then the stochastic integral in \eqref{sect-2-def-2.1} makes sense. In fact, since $\|\sigma_t(x)\|^2=\mbox{Tr}(a_t(x))\leq \|a_t(x)\|$,
  $$\E\int_0^T \|\sigma_t(X_t)\|^2\,\d t =\int_0^T \!\!\int_{\R^d} \|\sigma_t(x)\|^2\,\d\mu_t(x)\d t\leq \int_0^T \!\!\int_{\R^d} \|a_t(x)\| \,\d\mu_t(x)\d t<\infty.$$
Therefore, $t\mapsto \int_0^t \sigma(X_s)\,\d W_s$ is a square integrable martingale.

The assertion below is a special case of \cite[Chap. IV, Proposition 2.1]{Ikeda}.

\begin{proposition}[Existence of martingale solution implies that of weak solution]\label{2-prop-1}
Let $\mu_0\in\mathcal P(\R^d)$ and $P_{\mu_0}$  be a martingale solution of SDE \eqref{Ito-SDE}. Then there exists a weak solution $(\Omega,\mathcal G,(\mathcal G_t)_{0\leq t\leq T},P; X,W)$ to SDE \eqref{Ito-SDE} such that $X_\# P=P_{\mu_0}$.
\end{proposition}

The next result is similar to Lemma \ref{2-lem}; see also the proof of \cite[Chap. IV, Theorem 1.1]{Ikeda}.

\begin{proposition}\label{2-prop-2}
For each $i\in\{1,2\}$, let $\big(\Omega^i,\mathcal G^i, (\mathcal G_t^i)_{0\leq t\leq T},P^i; X^i,W^i\big)$ be a weak solution to SDE \eqref{Ito-SDE} with corresponding coefficients $\sigma^i$ and $b^i$, having the initial law $\mu^i_0\in\mathcal P(\R^d)$. Then for any $\pi\in \mathcal C(\mu^1_0,\mu^2_0)$, there exist a filtered probability space $(\Omega,\mathcal G,(\mathcal G_t)_{0\leq t\leq T},P)$, a standard $m$-dimensional $(\mathcal G_t)$-Brownian motion $W_t$ and two $\R^d$-valued $(\mathcal G_t)$-adapted continuous processes $Y^1$ and $Y^2$, such that $(Y^1_0, Y^2_0)$ is distributed as  $\pi$, and for each $i\in\{1,2\}$,
\begin{itemize}
\item[\rm(1)] $X^i$ and $Y^i$ have the same distributions in $\mathbb W_T^d$;
\item[\rm(2)] $\big(\Omega,\mathcal G, (\mathcal G_t)_{0\leq t\leq T},P;Y^i,W\big)$ is a weak solution of SDE \eqref{Ito-SDE} with coefficients $\sigma^i$ and $b^i$.
\end{itemize}
\end{proposition}

Finally we recall the following superposition principle proved in \cite[Theorem 2.5]{Trevisan} for solutions to the Fokker--Planck equation \eqref{FPE}, which is much more general than the one in \cite[Theorem 2.6]{Figalli} where the coefficients are required to be uniformly bounded.

\begin{proposition}\label{prop-Trevisan}
Given $\mu_0\in\mathcal P(\R^d)$, let $\mu_t\in\mathcal P(\R^d)$ be a measure-valued weak solution to equation \eqref{FPE} with initial value $\mu_0$, that is, it satisfies \eqref{integrability} and \eqref{FPE-1}. Then there exists a martingale solution $P_{\mu_0}$ to SDE \eqref{Ito-SDE} with initial law $\mu_0$ such that, for all $\varphi\in C_c^\infty(\R^d)$, one has
  $$\int_{\R^d} \varphi(x)\,\d\mu_t(x)=\int_{\mathbb W_T^d}\varphi(w_t)\,\d P_{\mu_0}(w),   \quad \forall\, t\in[0,T].$$
\end{proposition}

Before concluding this section, we make some remarks. Let $(\mu_t)_{t\in [0,T]}$ be a weak solution to the Fokker--Planck equation \eqref{FPE} with initial value $\mu_0$. By Propositions \ref{prop-Trevisan} and \ref{2-prop-1}, there exists a weak solution $(\Omega,\mathcal G,(\mathcal G_t)_{0\leq t\leq T},P; X,W)$ to the SDE \eqref{Ito-SDE}, that is, $\mbox{law}(X_0)=\mu_0$ and \eqref{sect-2-def-2.1} holds for a.s. $\omega\in \Omega$.

\begin{lemma}\label{2-lem-growth}
Let $(\mu_t)_{t\in[0,T]}$ be a solution to the Fokker--Planck equation \eqref{FPE}. We have the following simple estimates:
\begin{itemize}
\item[\rm(1)] if $\int_{\R^d} |x|\,\d\mu_0(x)<+\infty$, then $\sup_{0\leq t\leq T} \int_{\R^d} |x|\,\d\mu_t(x)<+\infty$;
\item[\rm(2)] if $\int_{\R^d} \log(1+|x|^2)\,\d\mu_0(x)<+\infty$, then $\sup_{0\leq t\leq T} \int_{\R^d} \log(1+|x|^2)\,\d\mu_t(x)<+\infty$.
\end{itemize}
\end{lemma}

\begin{proof}
(1) Since the solution $(\mu_t)_{0\leq t\leq T}$ of  \eqref{FPE} satisfies \eqref{integrability}, by the remark after Definition \ref{sect-2-def-2}, the stochastic integral in \eqref{sect-2-def-2.1} is a square integrable martingale. Hence, by the Burkholder inequality and Cauchy inequality,
  $$\aligned
  \E\sup_{0\leq t\leq T}|X_t| &\leq \E|X_0| +\E\sup_{t\leq T}\bigg|\int_0^t \sigma_s(X_s)\,\d W_s\bigg| +\E\sup_{t\leq T}\bigg|\int_0^t b_s(X_s)\,\d s\bigg|\\
  &\leq \int_{\R^d} |x|\,\d\mu_0(x) + C\bigg[\E\int_0^T \|\sigma_s(X_s)\|^2\,\d s\bigg]^{1/2} + \int_0^T \E |b_s(X_s)|\,\d s\\
  &= \int_{\R^d} |x|\,\d\mu_0(x) + C\bigg[\int_0^T\!\!\int_{\R^d} \|\sigma_s(x)\|^2\,\d\mu_s(x)\d s\bigg]^{1/2} + \int_0^T\!\!\int_{\R^d} |b_s(x)|\,\d\mu_s(x)\d s,
  \endaligned$$
which is finite by \eqref{integrability} since $\|\sigma_s(x)\|^2\leq \|a_s(x)\|$. Therefore,
  $$\sup_{0\leq t\leq T} \int_{\R^d} |x|\,\d\mu_t(x)=\sup_{0\leq t\leq T} \E|X_t|\leq \E\sup_{0\leq t\leq T}|X_t| <\infty.$$

(2) By the It\^o formula,
  $$\d\log(1+|X_t|^2)=\frac{2\<X_t,\sigma_t(X_t)\,\d W_t\>}{1+|X_t|^2} + \frac{2\<X_t,b_t(X_t)\>+\|\sigma_t(X_t)\|^2}{1+|X_t|^2}\,\d t - \frac{2|\sigma_t(X_t)^\ast X_t|^2}{(1+|X_t|^2)^2}\,\d t.$$
The quadratic variation of the martingale part is
  $$\E\int_0^T \frac{4|\sigma_t(X_t)^\ast X_t|^2}{(1+|X_t|^2)^2}\,\d t\leq \,\E\int_0^T \|\sigma_t(X_t)\|^2\,\d t\leq \int_0^T\!\!\int_{\R^d} \|a_t(x)\|\,\d\mu_t(x)\d t<\infty,$$
and hence, it is a square integrable martingale. Analogous to the above arguments, we have
  $$\aligned
  \E\sup_{0\leq t\leq T} \log(1+|X_t|^2)&\leq \E\log(1+|X_0|^2) +\E\sup_{0\leq t\leq T} \bigg|\int_0^t \frac{2\<X_t,\sigma_t(X_t)\,\d W_t\>}{1+|X_t|^2} \bigg|\\
  &\hskip13pt +\E\sup_{0\leq t\leq T}\bigg|\int_0^t \frac{2\<X_t,b_t(X_t)\>+\|\sigma_t(X_t)\|^2}{1+|X_t|^2}\,\d t\bigg|,
  \endaligned$$
which is dominated by
\begin{eqnarray*}&&\int_{\R^d} \log(1+|x|^2)\,\d\mu_0(x) + \bigg[\int_0^T\!\!\int_{\R^d} \|a_t(x)\|\,\d\mu_t(x)\d t\bigg]^{1/2} \\
&&+ \int_0^T\!\!\int_{\R^d} \big(\|a_t(x)\|+|b_t(x)|\big)\,\d\mu_t(x)\d t.
\end{eqnarray*}
This immediately implies the desired result.
\end{proof}

\section{Proofs of results in Section 2.1}

This section consists of four subsections. In the first three subsections we present the proofs of Theorems \ref{2-thm-1}, \ref{2-thm-2} and \ref{4-thm}, respectively. The proofs of the applications of the main results are given in Subsection 4.4.

\subsection{Proof of Theorem \ref{2-thm-1}}

Let $i\in\{1,2\}$. By Proposition \ref{prop-Trevisan}, there exists a martingale solution $P^i_{\mu^i_0}$ to the SDE \eqref{Ito-SDE} with coefficients $\sigma^i$ and $b^i$, and the initial probability distribution $\mu^i_0$ such that, for all $\varphi\in C_c^\infty(\R^d)$,
  $$\int_{\R^d}\varphi(x)u^i_t(x)\,\d x=\int_{\mathbb W_T^d}\varphi(w_t)\,\d P^i_{\mu^i_0}(w).$$
Applying Proposition \ref{2-prop-1}, we obtain a weak solution $\big(\Omega^i,\mathcal G^i, (\mathcal G_t^i)_{0\leq t\leq T},P^i; X^i, W^i\big)$ to SDE \eqref{Ito-SDE} with coefficients $\sigma^i$ and $b^i$, satisfying $(X^i)_\# P^i=P^i_{\mu^i_0}$. Next, we can find  $\pi_\delta\in \mathcal C(\mu^1_0,\mu^2_0)$ such that
  \begin{equation}\label{2-thm-1.0}
  \tilde \D_\delta (\mu^1_0,\mu^2_0)=\int_{\R^d\times\R^d} \log\bigg(1+\frac{|x-y|^2}{\delta^2}\bigg) \d\pi_\delta(x,y).
  \end{equation}
Finally, by Proposition \ref{2-prop-2}, there exists a common filtered probability space $(\Omega,\mathcal G,(\mathcal G_t)_{0\leq t\leq T},P)$, on which there are defined a standard $m$-dimensional $(\mathcal G_t)$-Brownian motion $W$ and two continuous $(\mathcal G_t)$-adapted processes $Y^1$ and $Y^2$ such that $\mbox{law}(Y^1_0, Y^2_0)= \pi_\delta$ and for $i=1,2$, $Y^i$ is distributed as $P^i_{\mu^i_0}$ on $\mathbb W_T^d$; moreover, it holds a.s. that
  $$Y^i_t=Y^i_0+\int_0^t b^i_s(Y^i_s)\,\d s  +\int_0^t \sigma^i_s(Y^i_s)\,\d W_s,\quad \mbox{for all } t\in[0, T].$$

The following arguments are by now standard for dealing with SDEs with weakly differentiable coefficients, see for instance \cite[Lemma 6.1]{Zhang10}, \cite[Theorem 5.2]{FangLuoThalmaier} and \cite[Lemma 4.1]{Zhang13}. Set $Z_t=Y^1_t-Y^2_t$ and fix $\delta>0$. We have by It\^o's formula that
  \begin{equation}\label{2-thm-1.1}
  \begin{split}
  \d\log\bigg(\frac{|Z_t|^2}{\delta^2}+1\bigg)
  &=2\frac{\big\< Z_t,\big[\sigma^1_t(Y^1_t) -\sigma^2_t(Y^2_t)\big]\,\d W_t\big\>} {|Z_t|^2+\delta^2} -2\frac{\big|\big[\sigma^1_t(Y^1_t)  -\sigma^2_t(Y^2_t)\big]^\ast Z_t\big|^2} {(|Z_t|^2+\delta^2)^2}\,\d t\\
  &\hskip13pt +\frac{2\big\< Z_t, b^1_t(Y^1_t)-b^2_t(Y^2_t)\big\>  +\big\|\sigma^1_t(Y^1_t) -\sigma^2_t(Y^2_t)\big\|^2}{|Z_t|^2+\delta^2}\,\d t.
  \end{split}
  \end{equation}
The quadratic variation of the martingale part on $[0,T]$ is finite, since, by \eqref{integrability},
  \begin{equation*}
  \begin{split}
  &\hskip13pt 4\,\E\int_0^T \frac{\big|\big[\sigma^1_t(Y^1_t) -\sigma^2_t(Y^2_t)\big]^\ast Z_t\big|^2} {(|Z_t|^2+\delta^2)^2}\,\d t\\ &\leq \frac4{\delta^2} \E\int_0^T \big\|\sigma^1_t(Y^1_t) -\sigma^2_t(Y^2_t)\big\|^2\,\d t\\
  &\leq \frac8{\delta^2} \int_0^T \bigg(\int_{\R^d} \|a^1_t\|\,\d\mu^1_t + \int_{\R^d} \|a^2_t\|\,\d\mu^2_t\bigg)\d t<\infty.
  \end{split}
  \end{equation*}
Hence, it is a square integrable martingale. Taking expectation on both sides of \eqref{2-thm-1.1} with respect to $P$ yields
  \begin{equation*}
  \begin{split}
  \E\log\bigg(\frac{|Z_t|^2}{\delta^2}+1\bigg)
  &\leq \E\log\bigg(\frac{|Z_0|^2}{\delta^2}+1\bigg)+ 2\,\E\int_0^{t}\frac{\big\langle Z_s,b^1_s(Y^1_s)-b^2_s(Y^2_s)\big\rangle} {|Z_s|^2+\delta^2}\,\d s\\
  &\hskip13pt +\E \int_0^t \frac{\big\|\sigma^1_s(Y^1_s) -\sigma^2_s(Y^2_s)\big\|^2} {|Z_s|^2+\delta^2}\,\d s.
  \end{split}
  \end{equation*}
Noticing that the joint distribution of $(Y^1_t, Y^2_t)$ belongs to $\mathcal C(\mu^1_t, \mu^2_t)$, we deduce from \eqref{2-thm-1.0} that
  \begin{equation}\label{2-thm-1.2}
  \aligned
  \tilde\D_\delta (\mu^1_t, \mu^2_t)&\leq \tilde\D_\delta (\mu^1_0, \mu^2_0) + 2\,\E\int_0^{t}\frac{\big\langle Z_s,b^1_s(Y^1_s)-b^2_s(Y^2_s)\big\rangle} {|Z_s|^2+\delta^2}\,\d s\\
  &\hskip13pt +\E \int_0^t \frac{\big\|\sigma^1_s(Y^1_s) -\sigma^2_s(Y^2_s)\big\|^2} {|Z_s|^2+\delta^2}\,\d s\\
  &=: \tilde\D_\delta (\mu^1_0, \mu^2_0) + I_1 +I_2. \endaligned
  \end{equation}
In the sequel, we shall estimate the two terms $I_1$ and $I_2$ separately.

\emph{Step 1.} By the triangle inequality, we have
  \begin{equation}\label{2-thm-1.3}
  \begin{split}
  I_1&\leq 2\,\E\int_0^t \frac{\big|b^1_s(Y^1_s) -b^2_s(Y^2_s)\big|} {\sqrt{|Z_s|^2+\delta^2}}\,\d s\\
  &\leq 2\,\E\int_0^t \frac{\big|b^1_s(Y^1_s) -b^1_s(Y^2_s)\big|} {\sqrt{|Z_s|^2+\delta^2}}\,\d s + 2\,\E\int_0^t \frac{\big|b^1_s(Y^2_s) -b^2_s(Y^2_s)\big|} {\sqrt{|Z_s|^2+\delta^2}}\,\d s\\
  &=: I_{1,1}+I_{1,2}.
  \end{split}
  \end{equation}
We first estimate $I_{1,2}$. Recall that $Y^2_s$ has the same law as $X^2_s$, which is distributed as $u^2_s(x)\,\d x$. Thus, by H\"{o}lder's inequality,
  \begin{equation}\label{2-thm-1.4}
  \begin{split}
  I_{1,2}&\leq \frac2\delta \int_0^t\! \int_{\R^d} \big|b^1_s(x) -b^2_s(x)\big|u^2_s(x)\,\d x\d s\\
  &\leq \frac2\delta \int_0^t\! \|b^1_s -b^2_s\|_{L^p}\|u^2_s\|_{L^q}\,\d s\\
  &\leq \frac2\delta \|u^2\|_{L^\infty(L^q)} \|b^1-b^2\|_{L^1(L^p)}.
  \end{split}
  \end{equation}
Next, in order to estimate $I_{1,1}$, we choose $\chi\in C_c^\infty(\R^d,\R_+)$ such that $\supp(\chi)\subset B(1)$ and $\int_{\R^d}\chi(x)\,\d x=1$. For $\eps\in(0,1)$ let $\chi_\eps(x)=\eps^{-d}\chi(x/\eps),\,x\in\R^d$. Define $b^{1,\eps}_s=b^1_s\ast\chi_\eps$. Then for a.e. $s\in[0,T]$, $b^{1,\eps}_s\in C^\infty(\R^d)$ for every $\eps\in(0,1)$. Then, by the triangle inequality again,
  \begin{equation}\label{2-thm-1.5}
  \begin{split}
  I_{1,1}&\leq 2\,\E\int_0^t \frac{\big|b^{1,\eps}_s(Y^1_s) -b^{1,\eps}_s(Y^2_s)\big|}{\sqrt{|Z_s|^2+\delta^2}}\, \d s\\
  &\hskip13pt +2\,\E \int_0^t \frac{\big|b^{1,\eps}_s(Y^1_s)- b^1_s(Y^1_s)\big| +\big|b^{1,\eps}_s(Y^2_s)- b^1_s(Y^2_s)\big|} {|Z_s|^2+\delta^2}\,\d s\\
  &=: I_{1,1,1}+I_{1,1,2}.
  \end{split}
  \end{equation}
Since $b^{1,\eps}_s$ is smooth, \eqref{2-lem-maximal-funct.1} implies that
  $$\aligned I_{1,1,1}&\leq 2C_d\, \E\int_0^t \big(M|\nabla b^{1,\eps}_s|(Y^1_s) + M|\nabla b^{1,\eps}_s|(Y^2_s)\big)\,\d s\\
  &= 2C_d \int_0^t\!\int_{\R^d} M|\nabla b^{1,\eps}_s|(x) \big(u^1_s(x) + u^2_s(x)\big)\,\d x\d s.\endaligned$$
Then H\"older's inequality leads to
  \begin{equation}\label{2-thm-1.6}
  \begin{split}
  I_{1,1,1}&\leq 2C_d \int_0^t \big\| M|\nabla b^{1,\eps}_s|\big\|_{L^p} \big(\|u^1_s\|_{L^q} +\|u^2_s\|_{L^q}\big)\,\d s\\
  &\leq 2C_d \bigg(\sum_{i=1}^2\|u^i\|_{L^\infty (L^q)}\bigg)\int_0^t C_{d,p}\big\| \nabla b^{1,\eps}_s\big\|_{L^p}\,\d s\\
  &\leq C'_{d,p}\bigg(\sum_{i=1}^2\|u^i\|_{L^\infty (L^q)}\bigg) \big\| \nabla b^1\big\|_{L^1(L^p)},
  \end{split}
  \end{equation}
where in the second inequality we have used \eqref{2-lem-maximal-funct.2}. The quantity $I_{1,1,2}$ can be treated as follows:
  $$\aligned I_{1,1,2}&\leq \frac2{\delta^2} \E\int_0^t \big(\big|b^{1,\eps}_s(Y^1_s)- b^1_s(Y^1_s)\big| + \big|b^{1,\eps}_s(Y^2_s)- b^1_s(Y^2_s)\big|\big)\,\d s \\
  &=  \frac2{\delta^2} \int_0^t\! \int_{\R^d} \big|b^{1,\eps}_s(x)- b^1_s(x)\big| \big(u^1_s(x) +u^2_s(x)\big)\,\d x\d s. \endaligned$$
Again, by H\"older's inequality,
  \begin{equation}\label{2-thm-1.7}
  \begin{split}
  I_{1,1,2}&\leq \frac2{\delta^2} \int_0^t \big\|b^{1,\eps}_s - b^1_s\big\|_{L^p} \big(\|u^1_s\|_{L^q} +\|u^2_s\|_{L^q}\big)\,\d s\\
  &\leq \frac2{\delta^2} \bigg(\sum_{i=1}^2\|u^i\|_{L^\infty(L^q)}\bigg)  \int_0^t \big\|b^{1,\eps}_s - b^1_s\big\|_{L^p} \,\d s.
  \end{split}
  \end{equation}
Since $b^1\in L^1(0,T; L^p(\R^d,\R^d))$, the right hand side of \eqref{2-thm-1.7} vanishes as $\eps\to0$. Combining \eqref{2-thm-1.5}--\eqref{2-thm-1.7} and letting $\eps\to0$, we arrive at
  \begin{equation*}
   I_{1,1}\leq C'_{d,p}\bigg(\sum_{i=1}^2\|u^i\|_{L^\infty(L^q)}\bigg) \big\| \nabla b^1\big\|_{L^1(L^p)}.
  \end{equation*}
This estimate together with \eqref{2-thm-1.3} and \eqref{2-thm-1.4} gives us
  \begin{equation}\label{2-thm-1.8}
  \begin{split}
  I_1&\leq \frac2\delta \|u^2\|_{L^\infty( L^q)} \|b^1-b^2\|_{L^1(L^p)} + C'_{d,p}\bigg(\sum_{i=1}^2\|u^i\|_{L^\infty(L^q)}\bigg) \big\| \nabla b^1\big\|_{L^1(L^p)}.
  \end{split}
  \end{equation}

\emph{Step 2.} The treatment of $I_2$ is similar to that of $I_1$. First,
  \begin{equation}\label{2-thm-1.9}
  \begin{split}
  I_2&\leq 2\,\E \int_0^t \frac{\big\|\sigma^1_s(Y^1_s) -\sigma^1_s(Y^2_s)\big\|^2} {|Z_s|^2+\delta^2}\,\d s +2\,\E \int_0^t \frac{\big\|\sigma^1_s(Y^2_s) -\sigma^2_s(Y^2_s)\big\|^2} {|Z_s|^2+\delta^2}\,\d s\\
  &=: I_{2,1} +I_{2,2}.
  \end{split}
  \end{equation}
In the similar way as for $I_{1,2}$, we have
  \begin{equation}\label{2-thm-1.10}
  I_{2,2}\leq \frac2{\delta^2}\|u^2\|_{L^\infty(L^q)} \|\sigma^1-\sigma^2\|_{L^2(L^{2p})}^2.
  \end{equation}
Next, denoting by $\sigma^{1,\eps}_s=\sigma^1_s\ast \chi_\eps$ with the same $\chi_\eps$ as above, we have
  $$\aligned I_{2,1}&\leq 6\,\E \int_0^t \frac{\big\|\sigma^{1,\eps}_s(Y^1_s) -\sigma^{1,\eps}_s(Y^2_s)\big\|^2} {|Z_s|^2+\delta^2}\,\d s\\
  &\hskip13pt + 6\,\E \int_0^t \frac{\big\|\sigma^{1,\eps}_s(Y^1_s) -\sigma^1_s(Y^1_s)\big\|^2 + \big\|\sigma^{1,\eps}_s(Y^2_s) -\sigma^1_s(Y^2_s)\big\|^2} {|Z_s|^2+\delta^2}\,\d s\\
  &=: I_{2,1,1}+I_{2,1,2}. \endaligned$$
Following the arguments for $I_{1,1,1}$ and $I_{1,1,2}$, respectively, we can show that
  $$I_{2,1,1}\leq C''_{d,p} \bigg(\sum_{i=1}^2\|u^i\|_{L^\infty(L^q)}\bigg) \big\| \nabla \sigma^1\big\|_{L^2(L^{2p})}^2 $$
and
  $$I_{2,1,2}\leq \frac6{\delta^2} \bigg(\sum_{i=1}^2\|u^i\|_{L^\infty(L^q)}\bigg) \int_0^t \big\|\sigma^{1,\eps}_s -\sigma^1_s\big\|_{L^{2p}}^2\,\d s.$$
Combining the above three inequalities and letting $\eps\to0$, we arrive at
  \begin{equation}\label{2-thm-1.11}
  I_{2,1}\leq  C''_{d,p} \bigg(\sum_{i=1}^2\|u^i\|_{L^\infty(L^q)}\bigg) \big\| \nabla \sigma^1\big\|_{L^2(L^{2p})}^2.
  \end{equation}
We deduce from \eqref{2-thm-1.9}--\eqref{2-thm-1.11} that
  \begin{equation}\label{2-thm-1.12}
  I_2\leq \frac2{\delta^2}\|u^2\|_{L^\infty(L^q)}\|\sigma^1-\sigma^2\|_{L^2(L^{2p})}^2 + C''_{d,p} \bigg(\sum_{i=1}^2\|u^i\|_{L^\infty(L^q)}\bigg) \big\| \nabla \sigma^1\big\|_{L^2(L^{2p})}^2.
  \end{equation}

\emph{Step 3.} Substituting the estimates \eqref{2-thm-1.8}  and \eqref{2-thm-1.12} into \eqref{2-thm-1.2}, we obtain
  \begin{equation}\label{2-thm-final}
  \begin{split}
  \tilde\D_\delta (\mu^1_t, \mu^2_t)&\leq \tilde\D_\delta (\mu^1_0, \mu^2_0) +2\|u^2\|_{L^\infty(L^q)} \bigg[\frac1\delta \|b^1-b^2\|_{L^1(L^p)} + \frac1{\delta^2} \|\sigma^1-\sigma^2\|_{L^2(L^{2p})}^2\bigg]\\
  &\hskip13pt + C_{d,p} \bigg[\sum_{i=1}^2\|u^i\|_{L^\infty(L^q)}\bigg] \Big(\big\| \nabla b^1\big\|_{L^1(L^p)} + \big\| \nabla \sigma^1\big\|_{L^2(L^{2p})}^2 \Big),
  \end{split}
  \end{equation}
where the constant $C_{d,p}>0$ depends only on $d$ and $p$. The proof is completed.

\subsection{Proof of Theorem \ref{2-thm-2}}

The proof of Theorem \ref{2-thm-2} is more or less similar to that of Theorem \ref{2-thm-1}. The main difference is that we cannot apply the maximal inequality \eqref{2-lem-maximal-funct.2} when dealing with the coefficients $b^i$ since it has only the $W^{1,1}$ regularity, so we shall follow the idea of argument in \cite[Lemma 4]{Seis16a}; see also \cite[Theorem 4.1]{Jabin}.

We still have \eqref{2-thm-1.2}, i.e.,
  \begin{equation}\label{thm-1.3-1}
  \begin{split}
  \tilde\D_\delta (\mu^1_t, \mu^2_t)&\leq \tilde\D_\delta (\mu^1_0, \mu^2_0) + 2\,\E\int_0^{t}\frac{\big\langle Z_s, b^1_s(Y^1_s)- b^2_s(Y^2_s)\big\rangle} {|Z_s|^2+\delta^2}\,\d s\\
  &\hskip13pt +\E \int_0^t \frac{\big\|\sigma^1_s(Y^1_s) -\sigma^2_s(Y^2_s)\big\|^2} {|Z_s|^2+\delta^2}\,\d s\\
  & =: \tilde\D_\delta (\mu^1_0, \mu^2_0) +I_1+I_2.
  \end{split}
  \end{equation}
The method to  estimate $I_2$ is almost the same as before, since $\sigma^1_s\in W^{1,2}(\R^d,\mathcal{M}_{d,m})$ and then the maximal inequality \eqref{2-lem-maximal-funct.2} is applicable. Hence,
  \begin{equation}\label{2-thm-2.1}
  \begin{split}
  I_2&\leq \frac2{\delta^2}\|u^2\|_{L^\infty(L^\infty)}\|\sigma^1-\sigma^2\|_{L^2(L^2)}^2 + C_d \bigg(\sum_{i=1}^2\|u^i\|_{L^\infty(L^\infty)}\bigg) \big\| \nabla \sigma^1\big\|_{L^2(L^2)}^2.
  \end{split}
  \end{equation}
We treat the term $I_1$ in a similar way as in \emph{Step 1} in the proof of Theorem \ref{2-thm-1}. Indeed, similar to \eqref{2-thm-1.4} and \eqref{2-thm-1.7}, it holds that
  \begin{equation}\label{2-thm-2.2}
  I_{1,2}\leq \frac2\delta \|u^2\|_{L^\infty(L^\infty)} \|b^1-b^2\|_{L^1(L^1)},
  \end{equation}
and
  \begin{equation}\label{2-thm-2.3}
  I_{1,1,2}\leq \frac2{\delta^2} \bigg(\sum_{i=1}^2 \|u^i\|_{L^\infty(L^\infty)}\bigg) \int_0^t \big\|b^{1,\eps}_s - b^1_s\big\|_{L^1} \,\d s,
  \end{equation}
which tends to 0 as $\eps\to 0$.

It remains to estimate
$$I_{1,1,1}=2\,\E\int_0^t \frac{\big|b^{1,\eps}_s(Y^1_s) -b^{1,\eps}_s(Y^2_s)\big|}{\sqrt{|Z_s|^2+\delta^2}}\, \d s,$$
for which we need the following lemma (see \cite[Lemma 3.1]{Jabin} for a sketched proof and \cite[Lemma 3.6]{Zhang13} for a related result).

\begin{lemma}\label{2-lem-2}
There exists a constant $C_d>0$ such that, for any smooth function $f:\R^d\to\R$ and any $x, y \in\R^d$,
  \begin{equation}\label{2-lem-2.1}
  |f(x)-f(y)|\leq C_d\int_{B(x,y)} \bigg(\frac1{|x-z|^{d-1}} + \frac1{|y-z|^{d-1}}\bigg)|\nabla f(z)|\,\d z,
  \end{equation}
where $B(x,y)$ is the ball with center $(x+y)/2$ and diameter $|x-y|$. Moreover,
  \begin{equation}\label{2-lem-2.2}
  \int_{B(x,y)} \bigg(\frac1{|x-z|^{d-1}} + \frac1{|y-z|^{d-1}}\bigg) \d z\leq C'_d|x-y|.
  \end{equation}
\end{lemma}

We continue the proof of Theorem \ref{2-thm-2}. By Lemma \ref{2-lem-2},
  \begin{equation}\label{2-thm-2.4}
  \begin{split}
  I_{1,1,1}&\leq 2C_d\,\E \int_0^T \frac1{\sqrt{|Z_s|^2+\delta^2}}\int_{B(Y^1_s,Y^2_s)} \bigg(\frac1{|Y^1_s-z|^{d-1}} + \frac1{|Y^2_s-z|^{d-1}}\bigg) |\nabla b^{1,\eps}_s(z)|\,\d z\d s.
  \end{split}
  \end{equation}
Fix any $M>0$. Let $B_{s,M}:=\{z\in B(Y^1_s,Y^2_s):|\nabla b^{1,\eps}_s(z)|\leq M\}$ and $B_{s,M}^c:=B(Y^1_s,Y^2_s) \setminus B_{s,M}$. Then
  \begin{equation}\label{2-thm-2.5}
  \begin{split}
  J_1&:= \E \int_0^T \frac1{\sqrt{|Z_s|^2+\delta^2}}\int_{B_{s,M}} \bigg(\frac1{|Y^1_s-z|^{d-1}} + \frac1{|Y^2_s-z|^{d-1}}\bigg) |\nabla b^{1,\eps}_s(z)|\,\d z\d s\\
  &\leq M\, \E \int_0^T \frac1{\sqrt{|Z_s|^2+\delta^2}}\int_{B(Y^1_s,Y^2_s)} \bigg(\frac1{|Y^1_s-z|^{d-1}} + \frac1{|Y^2_s-z|^{d-1}}\bigg) \d z\d s\\
  &\leq M\, \E \int_0^T \frac1{\sqrt{|Z_s|^2+\delta^2}} \, C'_d |Y^1_s-Y^2_s|\,\d s\leq C'_d MT,
  \end{split}
  \end{equation}
where in the second inequality we have used \eqref{2-lem-2.2}. Next, since $|\nabla b^1|\in L^1(0,T; L^1(\R^d))$, by the de la Vall\'ee--Poussin theorem (see e.g. \cite[Theorem 22]{MeyerB}), we can find a convex increasing function $G:\R_+\to \R_+$ such that
  $$s\mapsto\frac{G(s)}s \mbox{ increases monotonely to infinity as }s\uparrow\infty,$$
and
  $$\int_0^T\!\!\int_{\R^d} G(|\nabla b^1_s(x)|)\,\d x\d s <+\infty.$$
As a consequence,
  \begin{equation*}
  \begin{split}
  J_2&:= \E \int_0^T \frac1{\sqrt{|Z_s|^2+\delta^2}}\int_{B_{s,M}^c} \bigg(\frac1{|Y^1_s-z|^{d-1}} + \frac1{|Y^2_s-z|^{d-1}}\bigg) |\nabla b^{1,\eps}_s(z)|\,\d z\d s\\
  &\leq \frac{M}{G(M)} \E \int_0^T\!\! \int_{B_{s,M}^c} \bigg(\frac1{|Y^1_s-z|^{d-1}} + \frac1{|Y^2_s-z|^{d-1}}\bigg) \frac{G(|\nabla b^{1,\eps}_s(z)|)}{\sqrt{|Z_s|^2+\delta^2}}\,\d z\d s.
  \end{split}
  \end{equation*}
Note that $|Y^1_s-z| \vee |Y^2_s-z|\leq |Z_s|$ for all $z\in B_M^c\subset B(Y^1_s,Y^2_s)$. Hence,
  \begin{equation}\label{2-thm-2.6}
  \begin{split}
  J_2 &\leq \frac{M}{G(M)}\sum_{i=1}^2 \E \int_0^T\!\! \int_{B_{s,M}^c} \frac1{|Y^i_s-z|^{d-1}\sqrt{|Y^i_s-z|^2+\delta^2}} G(|\nabla b^{1,\eps}_s(z)|)\,\d z\d s\\
  &\leq \frac{M}{G(M)}\sum_{i=1}^2 \int_0^T\!\! \int_{\R^d} G(|\nabla b^{1,\eps}_s(z)|) \E\bigg( \frac1{|Y^i_s-z|^{d-1}\sqrt{|Y^i_s-z|^2+\delta^2}}\bigg) \d z\d s\\
  &= \frac{M}{G(M)}\sum_{i=1}^2 \int_0^T\!\! \int_{\R^d} G(|\nabla b^{1,\eps}_s(z)|) \int_{\R^d} \frac{u^i_s(x)}{|x-z|^{d-1}\sqrt{|x-z|^2+\delta^2}}\,\d x \d z\d s.
  \end{split}
  \end{equation}
Using the facts that $u^i_s\in L^\infty(\R^d)$ and it is a probability density, we have
  $$\aligned
  &\hskip13pt\int_{\R^d} \frac{u^i_s(x)}{|x-z|^{d-1}\sqrt{|x-z|^2+\delta^2}}\,\d x\\
  &= \bigg(\int_{\{|x-z|\leq 1\}} +\int_{\{|x-z|> 1\}}\bigg) \frac{u^i_s(x)}{|x-z|^{d-1}\sqrt{|x-z|^2+\delta^2}}\,\d x\\
  &\leq \bar C_d \|u^i_s\|_{L^\infty} \int_0^1 \frac1{\sqrt{r^2+\delta^2}}\,\d r +1\\
  &\leq 1+ \sqrt 2 \bar C_d \|u^i_s\|_{L^\infty} \log\Big(1+\frac1\delta\Big).
  \endaligned$$
Substituting this estimate into \eqref{2-thm-2.6} leads to
  $$\aligned
  J_2&\leq \frac{M}{G(M)}\sum_{i=1}^2 \bigg[1+\sqrt 2\, \bar C_d \|u^i\|_{L^\infty(L^\infty)} \log\Big(1+\frac1\delta\Big)\bigg] \int_0^T\!\! \int_{\R^d} G(|\nabla b^{1,\eps}_s(z)|)\, \d z\d s\\
  &\leq \tilde C_d\frac{M}{G(M)}\bigg[1+ \log\Big(1+\frac1\delta\Big)\sum_{i=1}^2 \|u^i\|_{L^\infty(L^\infty)}\bigg] \int_0^T\!\! \int_{\R^d} G(|\nabla b^1_s|\ast\chi_\eps(z))\, \d z\d s.
  \endaligned$$
The convexity of $G(s)$ and Jensen's inequality imply that
  $$\int_{\R^d} G(|\nabla b^1_s|\ast\chi_\eps(z))\, \d z\leq \int_{\R^d} G(|\nabla b^1_s|)\ast\chi_\eps(z)\, \d z= \int_{\R^d} G(|\nabla b^1_s|)\, \d z.$$
Therefore,
  $$J_2\leq \tilde C_d\frac{M}{G(M)}\|G(|\nabla b^1|)\|_{L^1(L^1)}\bigg[1+ \log\Big(1+\frac1\delta\Big)\sum_{i=1}^2 \|u^i\|_{L^\infty(L^\infty)}\bigg].$$
Combining this estimate with \eqref{2-thm-2.4} and \eqref{2-thm-2.5}, we finally get
  $$I_{1,1,1}\leq C'_d MT + \tilde C_d\frac{M}{G(M)} \|G(|\nabla b^1|)\|_{L^1(L^1)} \bigg[1+ \log\Big(1+\frac1\delta\Big)\sum_{i=1}^2 \|u^i\|_{L^\infty(L^\infty)}\bigg].$$
Define
  $$\phi(\delta)=\inf_{M>0}\left\{M+\frac{M}{G(M)}\bigg[1+\log\Big(1+\frac{1}{\delta}\Big)\bigg]\right\}.$$
Then $\frac{\phi(\delta)}{\log (1/\delta)}\to 0$ as $\delta$ vanishes. Indeed, for any $M>0$ and $\delta \in (0,1)$,
  $$\frac{\phi(\delta)}{\log \frac1\delta}\leq \frac{M}{\log \frac1\delta} +\frac{M}{G(M)} \cdot \frac{1+ \log\big(1+\frac{1}{\delta}\big)}{\log \frac1\delta}.$$
Since $G(s)/s$ tends to $\infty$ as $s\uparrow \infty$, first letting $\delta\to 0$ and then $M\to \infty$ gives the result. With this notation, we obtain
  \begin{eqnarray}\label{estimate-I_1,1,1}
  I_{1,1,1} \leq C_{d,T} \phi(\delta) \big(1+\|G(|\nabla b^1|)\|_{L^1(L^1)}\big) \bigg[1+ \sum_{i=1}^2 \|u^i\|_{L^\infty(L^\infty)}\bigg],
  \end{eqnarray}
where the constant $C_{d,T}>0$, depending only on $d$ and $T$.

Finally, as $I_1\leq I_{1,1}+I_{1,2}\leq I_{1,1,1}+I_{1,1,2}+I_{1,2}$, we combine \eqref{estimate-I_1,1,1} together with \eqref{2-thm-2.2} and \eqref{2-thm-2.3} and let $\eps\to0$ to get that
  $$\aligned
  I_1&\leq \frac2\delta \|u^2\|_{L^\infty(L^\infty)} \|b^1-b^2\|_{L^1(L^1)}\\
  &\hskip13pt + C_{d,T} \phi(\delta) \big(1+\|G(|\nabla b^1|)\|_{L^1(L^1)}\big) \bigg[1+ \sum_{i=1}^2 \|u^i\|_{L^\infty(L^\infty)}\bigg].
  \endaligned$$
With the above inequality  and \eqref{2-thm-2.1} in mind, we finish the proof in a similar way as in \emph{Step 3} of the proof of Theorem \ref{2-thm-1}.

\subsection{Proof of Theorem \ref{4-thm}}

We take $\pi_\delta\in \mathcal C(\mu^1_0, \mu^2_0)$ such that
  \begin{equation}\label{proof.1}
  \D_{\psi_\delta} (\mu^1_0, \mu^2_0)= \int_{\R^d\times \R^d} \psi_\delta (|x-y|^2)\,\d\pi_\delta(x,y).
  \end{equation}
As in the proof of Theorem \ref{2-thm-1}, there is a filtered probability space $(\Omega,\mathcal{G},(\mathcal{G}_t)_{0\leq t\leq T},P)$, on which there are defined a standard $m$-dimensional $(\mathcal{G}_t)$-Brownian motion $W$ and two continuous $(\mathcal{G}_t)$-adapted processes $Y^1$ and $Y^2$ such that $\pi_\delta =\mbox{law}(Y^1_0, Y^2_0)$ and, for each $i=1,2$,   $Y^i$ is distributed as $P^i_{\mu^i_0}$ on $\mathbb W^d_T$; moreover, it holds a.s. that
  $$Y^i_t=Y^i_0+\int_0^t b^i_s(Y^i_s)\,\d s +\int_0^t \sigma^i_s(Y^i_s)\,\d W_s,\quad \mbox{for all } 0\leq t\leq T.$$

Set $Z_t=Y^1_t-Y^2_t$ and fix $\delta>0$. Recall that
 $$\psi_\delta(s)=\int_0^s \frac{\d r}{\rho(r) +\delta^2},\quad s>0.$$
We have by It\^o's formula that
  \begin{equation*}
  \begin{split}
  \psi_\delta \big(|Z_t|^2\big)
  &= \psi_\delta(|Z_0|^2)+ 2\int_0^t \frac{\big\langle Z_s,\big[\sigma^1_s(Y^1_s) -\sigma^2_s(Y^2_s)\big]\,\d W_s\big\rangle} {\rho(|Z_s|^2)+\delta^2}\\
  &\hskip13pt +\int_0^t \frac{2\big\langle Z_s,b^1_s(Y^1_s)-b^2_s(Y^2_s)\big\rangle
  +\big\|\sigma^1_s(Y^1_s)-\sigma^2_s(Y^2_s)\big\|^2}{\rho(|Z_s|^2)+\delta^2}\,\d s\\
  &\hskip13pt -2\int_0^t \rho'(|Z_s|^2)\frac{\big|\big[\sigma^1_s(Y^1_s) -\sigma^2_s(Y^2_s)\big]^\ast Z_s\big|^2} {({\rho(|Z_s|^2)+\delta^2})^2}\,\d s.
  \end{split}
  \end{equation*}
Similar to the arguments right below \eqref{2-thm-1.1}, we can show that the martingale part is a square integrable martingale, since $\rho(s)\geq s\geq 0$. Using the fact that $\rho'\geq 0$ and taking expectation on both sides with respect to $P$, we derive that
  \begin{equation*}
  \begin{split}
  \E\psi_\delta \big(|Z_t|^2\big)
  &\leq \E\psi_\delta(|Z_0|^2)+ \E\int_0^t \frac{2\big\< Z_s, b^1_s(Y^1_s)-b^2_s(Y^2_s)\big\>+ \big\|\sigma^1_s(Y^1_s) -\sigma^2_s(Y^2_s)\big\|^2} {\rho(|Z_s|^2)+\delta^2}\,\d s.
  \end{split}
  \end{equation*}
Since $\mbox{law}(Y^1_0, Y^2_0)\in \mathcal C(\mu^1_t, \mu^2_t)$ for $t\in[0,T]$, this inequality plus \eqref{proof.1} leads to
  \begin{equation}\label{proof.2}
  \begin{split}
  \D_{\psi_\delta} (\mu^1_t, \mu^2_t)
  &\leq \D_{\psi_\delta} (\mu^1_0, \mu^2_0)+ 2\,\E\int_0^t \frac{\big\langle Z_s, b^1_s(Y^1_s)-b^2_s(Y^2_s)\big\rangle} {\rho(|Z_s|^2)+\delta^2}\,\d s\\
  &\hskip13pt +\E \int_0^t \frac{\big\|\sigma^1_s(Y^1_s) -\sigma^2_s(Y^2_s)\big\|^2} {\rho(|Z_s|^2)+\delta^2}\,\d s\\
  & =: \D_{\psi_\delta} (\mu^1_0, \mu^2_0) +  I_1+I_2.
  \end{split}
  \end{equation}
We shall estimate the two terms $I_1$ and $I_2$ in the next two steps, respectively.

\emph{Step 1.}  The arguments are similar to \emph{Step 1} of the proof of Theorem \ref{2-thm-1}. We have
  \begin{equation}\label{proof.5.0}
  \begin{split}
  I_1&= 2\,\E\int_0^t \frac{\big\langle Z_s, b^1_s(Y^1_s)-b^1_s(Y^2_s)\big\rangle} {\rho(|Z_s|^2)+\delta^2}\,\d s +2\,\E\int_0^t \frac{\big\langle Z_s, b^1_s(Y^2_s)-b^2_s(Y^2_s)\big\rangle} {\rho(|Z_s|^2)+\delta^2}\,\d s\\
  &=: I_{1,1}+I_{1,2}.
  \end{split}
  \end{equation}
For $I_{1,2}$, since $\rho(s)\geq s\geq 0$, we have
  \begin{equation}\label{proof.5.5}
  \begin{split}
  I_{1,2}&\leq 2\, \E\int_0^t\frac{|b_s^1(Y^2_s)-b^2_s(Y^2_s)|}{\sqrt{\rho(|Z_s|^2)+\delta^2}}\,\d s \\
  &\leq  \frac{2}{\delta} \int_0^t\!\! \int_{\R^d} |b_s^1(x)-b^2_s(x)|u^2_s(x)\,\d x \d s\\
  &\leq \frac{2}{\delta}\|u^2\|_{L^\infty(L^\infty)}\|b^1-b^2\|_{L^1(L^1)}.
  \end{split}
  \end{equation}
Next,
  \begin{equation}\label{proof.5}
  \begin{split}
  I_{1,1} &= 2\,\E\int_0^t \frac{\big\<Z_s, b^{1,\eps}_s(Y^1_s)- b^{1,\eps}_s(Y^2_s)\big\>} {\rho(|Z_s|^2)+\delta^2}\,\d s\\
  &\hskip13pt + 2\,\E\int_0^t \frac{\big\<Z_s,b^1_s(Y^1_s) -b^{1,\eps}_s(Y^1_s)\big\>  +\big\<Z_s,b^{1,\eps}_s(Y^2_s) - b^1_s(Y^2_s)\big\>}{\rho(|Z_s|^2)+\delta^2}\,\d s\\
  &=: I_{1,1,1}+I_{1,1,2}.
  \end{split}
  \end{equation}
Using again the fact that $\rho(s)\geq s\geq 0$, we have
  \begin{equation}\label{proof.5+1}
  \begin{split}
  I_{1,1,2} &\leq 2\sum_{i=1}^2 \E\int_0^t \frac{\big|b^{1,\eps}_s(Y^i_s) -b^1_s(Y^i_s)\big|}{\sqrt{\rho(|Z_s|^2)+\delta^2}} \,\d s\\
  &\leq \frac 2{\delta} \sum_{i=1}^2 \int_0^t\! \int_{\R^d} |b^{1,\eps}_s(x)-b^1_s(x)|u^i_s(x)\,\d x \d s\\
  &\leq \frac 2{\delta} \sum_{i=1}^2 \|u^i\|_{L^\infty(L^\infty)} \int_0^t\|b^{1,\eps}_s- b^1_s\|_{L^1}\,\d s.
  \end{split}
  \end{equation}
Since $b^1\in L^1(0,T; L^1(\R^d,\R^d))$, by the dominated convergence theorem, the right hand side of \eqref{proof.5+1} tends to 0 as $\varepsilon\rightarrow0$.

Now, we deal with the term $I_{1,1,1}$ and we shall use the hypothesis $(\mathbf{H})$. For any $s\in [0,T]$, there exists a negligible set $N_s\subset \R^d$ such that for all $x,y\in \R^d\setminus N_s$, we have
  $$|\<x-y,b^1_s(x)-b^1_s(y)\>| \leq \big(g_s(x) + g_s(y)\big) \rho(|x-y|^2).$$
Fix any $x_0, y_0\in\R^d$. Note that $(x_0-N_s)\cup (y_0-N_s)$ is a negligible set. For any $z\notin (x_0-N_s)\cup (y_0-N_s)$, one has $x_0-z\notin N_s$ and $y_0-z\notin N_s$. Thus
  \begin{equation}\label{proof.5+2}
  \begin{split}
  |\<x_0-y_0, b^{1,\eps}_s(x_0)- b^{1,\eps}_s(y_0)\>| &\leq \int_{\R^d} |\<x_0-y_0, b^1_s(x_0-z)- b^1_s(y_0-z)\>|\chi_\eps(z)\,\d z\\
  &\leq \int_{\R^d} \big(g_s(x_0-z) + g_s(y_0-z)\big) \rho(|x_0-y_0|^2) \chi_\eps(z)\,\d z\\
  &= \big(g_s^\eps(x_0) + g_s^\eps(y_0)\big) \rho(|x_0-y_0|^2),
  \end{split}
  \end{equation}
where $g_s^\eps =g_s\ast\chi_\eps$. Consequently,
  \begin{align*}
  I_{1,1,1}&\leq 2\,\E \int_0^t \big[g^\eps_s(Y^1_s)+ g^\eps_s(Y^2_s)\big]\,\d s.
  \end{align*}
Then, analogous to the above calculations,
  \begin{align*}
  I_{1,1,1}&\leq 2\sum_{i=1}^2 \int_0^{t}\!\! \int_{\R^d} g^\eps_s(x) u^i_s(x)\,\d x\d s \leq 2 \|g\|_{L^1(L^1)} \sum_{i=1}^2 \big\|u^i \big\|_{L^\infty(L^\infty)}.
  \end{align*}
This estimate together with \eqref{proof.5.0}--\eqref{proof.5+1} yields
  \begin{equation}\label{proof.7}
  I_1\leq \frac{2}{\delta}\|u^2\|_{L^\infty(L^\infty)}\|b^1-b^2\|_{L^1(L^1)} + 2\|g\|_{L^1(L^1)} \sum_{i=1}^2 \big\|u^i \big\|_{L^\infty(L^\infty)}.
  \end{equation}

\emph{Step 2.} We  deal now with the term $I_2$.
  \begin{equation}\label{proof.2.5}
  \begin{split}
  I_2&\leq 2\,\E \int_0^t \frac{\big\|\sigma^1_s(Y^1_s) -\sigma^1_s(Y^2_s)\big\|^2} {\rho(|Z_s|^2)+\delta^2}\,\d s + 2\,\E \int_0^t \frac{\big\|\sigma^1_s(Y^2_s) -\sigma^2_s(Y^2_s)\big\|^2} {\rho(|Z_s|^2)+\delta^2}\,\d s\\
  &=: I_{2,1}+I_{2,2}.
  \end{split}
  \end{equation}
Analogous to \eqref{2-thm-1.10}, we have
  \begin{equation}\label{proof.2.5.1}
  I_{2,2} \leq \frac{2}{\delta^2}\|u^2\|_{L^\infty(L^\infty)} \|\sigma^1 -\sigma^2\|_{L^2(L^2)}^2.
  \end{equation}

Let $\sigma^{1,\eps}=\sigma^1\ast\chi_\varepsilon$ be as above. For any $\eps>0$, we have
  \begin{equation}\label{proof.3}
  \begin{split}
  I_{2,1}&\leq 6\,\E \int_0^t \frac{\big\|\sigma^{1,\eps}_s(Y^1_s) -\sigma^{1,\eps}_s(Y^2_s)\big\|^2} {\rho(|Z_s|^2)+\delta^2}\,\d s\\
  &\hskip13pt +6\,\E \int_0^t \frac{\big\|\sigma^{1,\eps}_s(Y^1_s) -\sigma^1_s(Y^1_s)\big\|^2 +\big\|\sigma^{1,\eps}_s(Y^2_s)
  -\sigma^1_s(Y^2_s)\big\|^2} {\rho(|Z_s|^2)+\delta^2}\,\d s\\
  &=: I_{2,1,1}+I_{2,1,2}.
  \end{split}
  \end{equation}
The estimation of $I_{2,1,2}$ is similar as before:
  \begin{equation}\label{proof.3.0}
  \begin{split}
  I_{2,1,2}&\leq \frac 6{\delta^2}\sum_{i=1}^2\E\int_0^t \big\|\sigma^{1,\eps}_s(Y^i_s) -\sigma^1_s(Y^i_s)\big\|^2 \,\d s\\
  &= \frac 6{\delta^2}\sum_{i=1}^2 \int_0^t\!\! \int \|\sigma^{1,\eps}_s(x) -\sigma^1_s(x)\|^2 u^i_s(x)\,\d x \d s\\
  &\leq \frac {6}{\delta^2}\sum_{i=1}^2\|u^i\|_{L^\infty(L^\infty)} \int_0^t \|\sigma^{1,\eps}_s -\sigma^1_s\|^2_{L^2}\,\d s,
  \end{split}
  \end{equation}
which vanishes as $\eps\to 0$, since $\sigma^1\in L^2(0,T; L^{2}(\R^d,\mathcal{M}_{d,m}))$. Next, similar to \eqref{proof.5+2}, we have
  $$\|\sigma^{1,\eps}_s(x)- \sigma^{1,\eps}_s(x)\|^2\leq \big(g_s^\eps(x) + g_s^\eps(y)\big) \rho(|x-y|^2), \quad \mbox{for all } x,y\in \R^d.$$
Thus,
  \begin{align*}
  I_{2,1,1}&\leq 6\,\E \int_0^t \big[ g_s^\eps(Y^1_s) + g_s^\eps(Y^2_s)\big] \,\d s.
  \end{align*}
Recall that $Y^i_s$ is distributed as $u^i_s(x)\,\d x,\,i=1,2$. Consequently,
  \begin{align*}
  I_{2,1,1}&\leq 6\int_0^t \int_{\R^d} g_s^\eps(x) \big(u^1_s(x) +u^2_s(x)\big)\,\d x \d s \leq 6\sum_{i=1}^2\big\|u^i\big\|_{L^\infty(L^\infty)} \|g\|_{L^1(L^1)}.
  \end{align*}
Note that the upper bound is independent of $\eps>0$. Combining the above estimate with \eqref{proof.2.5}--\eqref{proof.3.0}, and letting $\eps\to 0$ on the right hand side of \eqref{proof.3}, we obtain
  \begin{equation}\label{proof.4}
  I_2\leq 6\sum_{i=1}^2\big\|u^i\big\|_{L^\infty(L^\infty)} \|g\|_{L^1(L^1)} + \frac{2}{\delta^2}\|u^2\|_{L^\infty(L^\infty)} \|\sigma^1 -\sigma^2\|_{L^2(L^2)}^2.
  \end{equation}

\emph{Step 3.} Combining \eqref{proof.2}, \eqref{proof.7} and \eqref{proof.4}, we finally obtain
  \begin{eqnarray*}
  \D_{\psi_\delta} (\mu^1_t, \mu^2_t)
  &\leq& \D_{\psi_\delta} (\mu^1_0, \mu^2_0) + 8\|g\|_{L^1(L^1)} \sum_{i=1}^2 \big\|u^i \big\|_{L^\infty(L^\infty)}\\
  &&+ 2\|u^2\|_{L^\infty(L^\infty)} \Big(\frac1\delta \|b^1-b^2\|_{L^1(L^1)}+ \frac1{\delta^2}\big\|\sigma^1 -\sigma^2\big\|^2_{L^2( L^2)}\Big).
  \end{eqnarray*}
This finishes the proof.

\subsection{Proofs of the other results}

\begin{proof}[Proof of Corollary \ref{cor-uniqueness}]
For $\delta=1/n$, we can find $\pi_n\in \mathcal C(\mu^1_t, \mu^2_t)$ such that
  \begin{equation}\label{cor-uniqueness.2}
  \tilde\D_{1/n}(\mu^1_t, \mu^2_t)= \int_{\R^d\times \R^d} \log\big(1+n^2|x-y|^2\big)\,\d\pi_n(x,y).
  \end{equation}
Since $\{\pi_n:n\geq 1\}\subset \mathcal C(\mu^1_t, \mu^2_t)$, it is clear that the family $\{\pi_n:n\geq 1\}$ is relatively compact with respect to the weak convergence. Up to a subsequence, we can assume that $\pi_n$ converges weakly to some probability measure $\pi_0$ on $\R^d\times \R^d$. It is easy to see that $\pi_0\in \mathcal C(\mu^1_t, \mu^2_t)$. We shall show that $\pi_0$ is supported on the diagonal of $\R^d\times \R^d$.

Fix an arbitrary $\kappa>0$. We define $E_\kappa=\{(x,y)\in \R^d\times \R^d: |x-y|>\kappa\}$ which is an open subset of $\R^d\times \R^d$. Summarizing the assertions of Theorems \ref{2-thm-1} and \ref{2-thm-2}, we can find a constant $C>0$ such that
  $$\int_{\R^d\times \R^d} \log\big(1+n^2|x-y|^2\big)\,\d \pi_n(x,y)= \tilde\D_{1/n}(\mu^1_t, \mu^2_t)\leq C[1+\phi(1/n)],$$
where $\phi(\delta) = o(|\log \delta|)$ as $\delta\to 0$. Therefore,
  $$\pi_n(E_\kappa) \log (1+n^2\kappa^2)\leq \int_{E_\kappa} \log\big(1+n^2|x-y|^2\big)\,\d\pi_n(x,y) \leq C[1+\phi(1/n)].$$
As $\pi_n$ converges weakly to $\pi_0$, we have
  $$\pi_0(E_\kappa)\leq \liminf_{n\to \infty} \pi_n(E_\kappa) \leq \liminf_{n\to \infty} \frac{C[1+\phi(1/n)]}{\log (1+n^2\kappa^2)}=0.$$
The arbitrariness of $\kappa>0$ implies that $\pi_0$ is supported on the diagonal of $\R^d\times \R^d$, i.e.,  for $\pi_0$-a.e. $(x,y)\in \R^d\times \R^d$, one has $x=y$. Now for any $\phi\in C_b(\R^d)$,
  $$\int_{\R^d} \phi(x)\,\d\mu^1_t(x)= \int_{\R^d\times \R^d} \phi(x)\,\d\pi_0(x,y) =\int_{\R^d\times \R^d} \phi(y)\,\d\pi_0(x,y) =\int_{\R^d} \phi(y)\,\d\mu^2_t(y),$$
and hence $\mu^1_t=\mu^2_t$.
\end{proof}

\begin{proof}[Proof of Proposition \ref{prop-zero-diffusivity}]
Recall that $C_{q,T}$ is the right hand side of \eqref{zero-diffusivity.1}. Since $\rho^\kappa_0=\rho^0_0=\bar\rho$, by Theorem \ref{2-thm-1},
  $$\tilde\D_\delta(\rho^\kappa_t,\rho^0_t)\leq \frac{2\kappa T}{\delta^2} C_{q,T}\|\sigma\|_{L^{2p}}^2 + 2 C_{d,p}C_{q,T}\|\nabla b\|_{L^1(L^p)}.$$
Therefore, taking $\delta=\sqrt\kappa$ leads to
  $$\tilde\D_{\sqrt\kappa}(\rho^\kappa_t,\rho^0_t)\leq 2 C_{q,T}\big[T\|\sigma\|_{L^{2p}}^2 + C_{d,p} \|\nabla b\|_{L^1(L^p)}\big],$$
for any $t\in[0,T]$. The proof is complete.
\end{proof}

We conclude the section by providing the

\begin{proof}[Proof of Theorem \ref{thm-variant}]
The proof is almost the same as that of  Theorem \ref{2-thm-1}, the only difference being that we apply H\"older's inequality to the diffusion and the drift parts with different exponents. More precisely, \eqref{2-thm-1.4} becomes
  \begin{equation*}
  I_{1,2}\leq \frac2\delta \int_0^t\! \|b^1_s -b^2_s\|_{L^{p_2}}\|u^2_s\|_{L^{p'_2}}\,\d s \leq \frac2\delta \|u^2\|_{L^\infty(L^{p'_2})} \|b^1-b^2\|_{L^1(L^{p_2})},
  \end{equation*}
where $p'_2$ is the conjugate number of $p_2$. Similarly, we rewrite \eqref{2-thm-1.6} as
  $$\aligned  I_{1,1,1}&\leq 2C_d \int_0^t \big\| M|\nabla b^{1,\eps}_s|\big\|_{L^{p_2}} \big(\|u^1_s\|_{L^{p'_2}} +\|u^2_s\|_{L^{p'_2}}\big)\,\d s\\
  &\leq C_{d,p_2}\bigg(\sum_{i=1}^2\|u^i_s\|_{L^\infty(L^{p'_2})}\bigg) \big\| \nabla b^1\big\|_{L^1(L^{p_2})},
  \endaligned$$
which leads to
  \begin{equation*}
   I_{1,1}\leq C_{d,p_2}\bigg(\sum_{i=1}^2\|u^i_s\|_{L^\infty(L^{p'_2})}\bigg) \big\| \nabla b^1\big\|_{L^1(L^{p_2})}.
  \end{equation*}
Combining the above estimates, we obtain
  \begin{equation}\label{thm-variant.1}
  I_1\leq \frac2\delta \|u^2\|_{L^\infty(L^{p'_2})} \|b^1-b^2\|_{L^1(L^{p_2})} + C_{d,p_2}\bigg(\sum_{i=1}^2\|u^i\|_{L^\infty(L^{p'_2})}\bigg) \big\| \nabla b^1\big\|_{L^1(L^{p_2})}.
  \end{equation}
In a similar way as Step 2 in the proof of Theorem \ref{2-thm-1}, we have
  \begin{equation}\label{thm-variant.2}
  I_2\leq \frac2{\delta^2}\|u^2\|_{L^\infty(L^{p'_1})}\|\sigma^1-\sigma^2\|_{L^2(L^{2p_1})}^2 + C_{d,p_1} \bigg(\sum_{i=1}^2\|u^i\|_{L^\infty(L^{p'_1})}\bigg) \big\| \nabla \sigma^1\big\|_{L^2(L^{2p_1})}^2.
  \end{equation}
Since the rest of the proof is the same, we omit it.
\end{proof}

\section{Proofs of results in Section 2.2}

This section is devoted to proving Theorems \ref{thm-LPS} and \ref{thm-LPS-Wass} for which we need some preparations. Consider the Fokker--Planck equation associated to \eqref{Ito-SDE-2}:
  \begin{equation}\label{FPE-LPS-1}
  \partial_t\nu_t-\frac12\sum_{ij}\partial_{ij} (\nu_t \tilde a_{ij} )+\div(\nu_t \tilde b)=0, \quad \nu|_{t=0}= \mbox{law}(Y_0),
  \end{equation}
where $\tilde a=\tilde \sigma\tilde\sigma^\ast$ with $\tilde\sigma$ and $\tilde b$ defined in \eqref{Ito-SDE-2.1}. We have the following simple result.

\begin{lemma}\label{lem-LPS}
Let $\psi_t$ be defined as in \eqref{Zvonkin.1}. Then the solutions of \eqref{FPE-LPS} and \eqref{FPE-LPS-1} have the following relations:
  $$\nu_t=(\psi_t)_\# \mu_t,\quad \mu_t=(\psi_t^{-1})_\# \nu_t.$$
Moreover, if $\d\mu_t=u_t\,\d x$ with $u\in L^\infty(0,T; L^r(\R^d))\, (r\in [1,\infty])$, then $\d\nu_t= v_t\,\d x$  and for some positive constant $C$,
  $$C^{-1} \|u\|_{L^\infty(L^r)}\leq \|v\|_{L^\infty(L^r)} \leq C \|u\|_{L^\infty(L^r)}.$$
\end{lemma}

\begin{proof}
Note that $\mu_t=\mbox{law}(X_t)$ and  $\nu_t=\mbox{law}(Y_t)$, where $X_t$ and $Y_t$ are respectively the solutions to the SDEs \eqref{Ito-SDE-1} and \eqref{Ito-SDE-2}. We have $Y_t=\psi_t(X_t)$, and hence for any $f\in C_b(\R^d)$,
  $$\int_{\R^d} f\,\d\nu_t= \E f(Y_t)=\E f(\psi_t(X_t)) =\int_{\R^d} f\circ\psi_t\,\d\mu_t = \int_{\R^d} f\, \d[(\psi_t)_\# \mu_t].$$
This implies the first identity. The second one can be proved analogously.

Next, it is easy to see that $v_t(y)=u_t(\psi_t^{-1}(y)) |\det(\nabla \psi_t^{-1}(y))|$. Recall that $\psi_t^{-1}$  has bounded first derivatives, uniformly in $t\in[0,T]$, hence the last assertion is obvious for $r=\infty$. If $r<\infty$, then
\begin{eqnarray*}\int_{\R^d} |v_t(y)|^r\,\d y&=&\int_{\R^d} |u_t(\psi_t^{-1}(y))|^r |\det(\nabla \psi_t^{-1}(y))|^r\,\d y\\
 &=& \int_{\R^d} |u_t(x)|^r |\det(\nabla \psi_t^{-1}(\psi_t(x)))|^{r-1}\,\d x,
 \end{eqnarray*}
where we have used the fact that $\det(\nabla \psi_t^{-1}(\psi_t(x))) \det(\nabla \psi_t(x))=1$ in the last equality.  Thus, $\|v\|_{L^\infty(L^r)} \leq C_r \|u\|_{L^\infty(L^r)}$. Similarly, we can prove the other inequality.
\end{proof}

Suppose we are given two vector fields $b^1, b^2\in L^q(0,T; L^p(\R^d,\R^d))$ with $p>2$ and $q>2$ such that $\frac dp +\frac 2q<1$. For each $i\in \{1,2\}$, let $\phi^i_t$ be the solution to \eqref{Zvonkin} with $b=b^i$, i.e.,
  \begin{equation}\label{5-LPS.1}
  \partial_t \phi^i_t+\frac12 \Delta \phi^i_t + b^i_t\cdot\nabla \phi^i_t-\lambda \phi^i_t = -b^i_t,\quad \phi^i_T(x)=0\in \R^d,
  \end{equation}
and define
  $$\psi^i_t(x)= x+\phi^i_t(x),\quad x\in \R^d,$$
which further gives us $\tilde\sigma^i$ and $\tilde b^i$ as in \eqref{Ito-SDE-2.1}, namely,
  \begin{equation*}
  \tilde\sigma_t^i(y)={\rm Id} + (\nabla \phi^i_t)\circ (\psi^i_t)^{-1}(y),\quad \tilde b^i_t(y)=\lambda \phi^i_t\circ (\psi^i_t)^{-1}(y),\quad y\in\R^d.
  \end{equation*}
By taking $\lambda>0$ big enough in \eqref{5-LPS.1}, we may assume (see e.g. \cite[Lemma 3.4]{Flandoli})
  \begin{equation}\label{LPS-0}
  \sup_{t\in [0,T]} \big(\|\nabla\phi^1_t\|_{L^\infty} \vee \|\nabla\phi^2_t\|_{L^\infty}\big)\leq \frac12.
  \end{equation}

\begin{lemma}\label{lem-LPS-0}
There exists some constant $C_{p,q,T}>0$ such that
  $$\sup_{t\in[0,T]} \|\phi^1_t -\phi^2_t\|_{W^{1,p}}\leq C_{p,q,T} \|b^1-b^2\|_{L^q(L^p)}.$$
\end{lemma}

\begin{proof}
Let $\varphi_t= \phi^1_t -\phi^2_t\ (0\leq t\leq T)$. Then
  $$\partial_t \varphi_t+\frac12 \Delta \varphi_t + b^1_t\cdot\nabla \varphi_t-\lambda \varphi_t = -(b^1_t -b^2_t)(\nabla\phi^2_t +{\rm Id} ),\quad \varphi_T(x)=0\in \R^d.$$
We have the following estimate which is analogous to \eqref{regularity}:
  \begin{equation}\label{regularity-1}
  \aligned
  \|\partial_t \varphi\|_{L^q(L^p)}+ \|\varphi\|_{L^q(W^{2,p})}&\leq C\|(b^1 -b^2)(\nabla\phi^2 +Id)\|_{L^q(L^p)}\\
  &\leq \tilde C \|b^1 -b^2\|_{L^q(L^p)},
  \endaligned
  \end{equation}
where the last inequality follows from \eqref{LPS-0}. Taking $\beta=1$ in \cite[Lemma 10.1]{Krylov}, we have
  $$\|\varphi_t\|_{W^{1,p}}\leq C(T-t)^{\frac12 -\frac1q} \|\varphi\|_{L^q(W^{2,p})}^{\frac12} \|\partial_t \varphi\|_{L^q(L^p)}^{\frac12} \leq \hat C (T-t)^{\frac12 -\frac1q} \|b^1 -b^2\|_{L^q(L^p)},$$
where we have used \eqref{regularity-1} in the last step. The proof is finished.
\end{proof}

The following estimates are crucial to the proofs of Theorems  \ref{thm-LPS} and \ref{thm-LPS-Wass}.

\begin{proposition}\label{prop-LPS}
We have
  $$\|\tilde b^1-\tilde b^2\|_{L^\infty(L^p)}\leq C \|b^1-b^2\|_{L^q(L^p)}$$
and
  $$\|\tilde \sigma^1-\tilde\sigma^2\|_{L^q(L^{p})}\leq C \|b^1-b^2\|_{L^q(L^p)}$$
for some constant $C>0$.
\end{proposition}

\begin{proof}
Note that $\tilde b^i_t(x) = \lambda \phi^i_t((\psi^i_t)^{-1}(x))\, (i=1,2)$. Then, by \eqref{LPS-0},
  \begin{equation}\label{prop-LPS.2}
  \aligned
  \|\tilde b^1_t- \tilde b^2_t\|_{L^p}&\leq \lambda\big\|\phi^1_t\circ (\psi^1_t)^{-1}-\phi^1_t \circ (\psi^2_t)^{-1}\big\|_{L^p} + \lambda\big\|\phi^1_t\circ (\psi^2_t)^{-1}-\phi^2_t \circ (\psi^2_t)^{-1}\big\|_{L^p}\\
  &\leq C\big\|(\psi^1_t)^{-1}- (\psi^2_t)^{-1}\big\|_{L^p} + C \big\|\phi^1_t -\phi^2_t\big\|_{L^p},
  \endaligned
  \end{equation}
where, for the second norm on the right hand side, we have used the change of variable formula and the fact that $\psi^2_t$ has bounded first derivatives. By the definition of $\psi^i_t$, we have
  $$y=\psi^i_t((\psi^i_t)^{-1}(y))= (\psi^i_t)^{-1}(y)+ \phi^i_t((\psi^i_t)^{-1}(y)),$$
thus
  \begin{equation}\label{prop-LPS.0}
  (\psi^i_t)^{-1}(y)=y- \phi^i_t((\psi^i_t)^{-1}(y)),\quad y\in \R^d,\ i=1,2.
  \end{equation}
As a result, by \eqref{LPS-0},
\begin{eqnarray*}
  &&|(\psi^1_t)^{-1}(y)-(\psi^2_t)^{-1}(y)|\\
  &\leq& |\phi^1_t((\psi^1_t)^{-1}(y)) -\phi^1_t((\psi^2_t)^{-1}(y))|+ |\phi^1_t((\psi^2_t)^{-1}(y)) -\phi^2_t((\psi^2_t)^{-1}(y))|\\
  &\leq& \frac12 |(\psi^1_t)^{-1}(y)-(\psi^2_t)^{-1}(y)| + |\phi^1_t((\psi^2_t)^{-1}(y)) -\phi^2_t((\psi^2_t)^{-1}(y))|.
  \end{eqnarray*}
Hence
  \begin{equation}\label{prop-LPS.0.5}
  |(\psi^1_t)^{-1}(y)-(\psi^2_t)^{-1}(y)|\leq  2|\phi^1_t((\psi^2_t)^{-1}(y)) -\phi^2_t((\psi^2_t)^{-1}(y))|,
  \end{equation}
which, by a similar treatment for the second norm on the right hand side of \eqref{prop-LPS.2}, leads to
  $$\big\|(\psi^1_t)^{-1}- (\psi^2_t)^{-1}\big\|_{L^p} \leq 2C\big\|\phi^1_t -\phi^2_t\big\|_{L^p}.$$
Substituting this estimate into  \eqref{prop-LPS.2} and applying Lemma \ref{lem-LPS-0}, we obtain the first estimate.

Next, since $\tilde \sigma^i_t(x)={\rm Id} + (\nabla\phi^i_t)((\psi^i_t)^{-1}(x))$, for $i=1,2$, we have
  \begin{equation}\label{prop-LPS.2.5}
  \aligned
  \|\tilde \sigma^1_t-\tilde\sigma^2_t\|_{L^p} &\leq \|(\nabla\phi^1_t)\circ (\psi^1_t)^{-1} -(\nabla\phi^1_t)\circ (\psi^2_t)^{-1}\|_{L^p}\\
  &\hskip13pt + \|(\nabla\phi^1_t)\circ (\psi^2_t)^{-1} -(\nabla\phi^2_t)\circ (\psi^2_t)^{-1}\|_{L^p}\\
  &=: I_1+I_2.
  \endaligned
  \end{equation}
We begin with the simpler term $I_2$. By the change of variable formula again, we have for all $t\in[0,T]$ that
  \begin{equation}\label{prop-LPS.3}
  \aligned
  I_2&= \bigg(\int_{\R^d} |\nabla\phi^1_t(x) -\nabla\phi^2_t(x)|^p |\det(\nabla \psi^2_t(x))|\,\d x\bigg)^{1/p}\\
  &\leq C \|\nabla\phi^1_t -\nabla\phi^2_t\|_{L^p} \leq C'\|b^1-b^2\|_{L^q(L^p)},
  \endaligned
  \end{equation}
where the last step follows from Lemma \ref{lem-LPS-0}.

We now deal with the term $I_1$, for which we need the pointwise characterization \eqref{2-lem-maximal-funct.1} of Sobolev functions. Therefore,
  $$I_1^p\leq C_d^p \int_{\R^d} |(\psi^1_t)^{-1}(x)-(\psi^2_t)^{-1}(x)|^p \big[M|\nabla^2\phi^1_t|((\psi^1_t)^{-1}(x)) + M|\nabla^2 \phi^1_t|((\psi^2_t)^{-1}(x))\big]^p\,\d x.$$
Similar to \eqref{prop-LPS.0.5}, we have
  $$|(\psi^1_t)^{-1}(x)-(\psi^2_t)^{-1}(x)|\leq  2|\phi^1_t((\psi^1_t)^{-1}(x)) -\phi^2_t((\psi^1_t)^{-1}(x))|.$$
Hence
  $$\aligned
  I_1^p&\leq C_{d,p} \int_{\R^d} |\phi^1_t((\psi^1_t)^{-1}(x)) -\phi^2_t((\psi^1_t)^{-1}(x))|^p \big[M|\nabla^2\phi^1_t|((\psi^1_t)^{-1}(x))\big]^p\,\d x\\
  &\hskip13pt +C_{d,p} \int_{\R^d} |\phi^1_t((\psi^2_t)^{-1}(x)) -\phi^2_t((\psi^2_t)^{-1}(x))|^p \big[M|\nabla^2\phi^1_t|((\psi^2_t)^{-1}(x))\big]^p\,\d x\\
  &\leq C_{d,p} \int_{\R^d} |\phi^1_t(y) -\phi^2_t(y)|^p \big[M|\nabla^2\phi^1_t|(y)\big]^p \big(|\det(\nabla\psi^1_t(y))| +|\det(\nabla\psi^2_t(y))|\big)\,\d y.
  \endaligned$$
Since $\psi^1_t$ and $\psi^1_t$ have bounded first derivatives, uniformly in $t\in[0,T]$, we arrive at
  \begin{equation}\label{prop-LPS.4}
  I_1^p \leq C' \int_{\R^d} |\phi^1_t(y) -\phi^2_t(y)|^p \big[M|\nabla^2\phi^1_t|(y)\big]^p \,\d y.
  \end{equation}
Applying \cite[Lemma 10.2(i)]{Krylov} with $\delta=1/2$ and noting that $\phi^1_T(y) -\phi^2_T(y)=0$, we obtain for all $y\in \R^d$ and $t\in [0,T]$ that
  $$|\phi^1_t(y) -\phi^2_t(y)|\leq C(T-t)^{\frac12} \|\phi^1 -\phi^2\|_{L^q(W^{2,p})}^{\frac12- \frac1q} \|\partial_t (\phi^1 -\phi^2)\|_{L^q(L^p)}^{\frac12+ \frac1q} \leq C_{T} \|b^1 -b^2\|_{L^q(L^p)},$$
where the last inequality is due to \eqref{regularity-1}. Substituting this estimate into \eqref{prop-LPS.4} yields
  $$I_1\leq C \|b^1 -b^2\|_{L^q(L^p)} \big\| M|\nabla^2\phi^1_t|\big\|_{L^p} \leq C'_p \|b^1 -b^2\|_{L^q(L^p)} \|\nabla^2\phi^1_t\|_{L^p}.$$
Note that $\|\nabla^2\phi^1\|_{L^q(L^p)} <\infty$. Combining the above inequality with \eqref{prop-LPS.2.5} and \eqref{prop-LPS.3}, we obtain the desired result.
\end{proof}

Let $\nu^i_t$ be the solution to \eqref{FPE-LPS-1} with $\tilde\sigma= \tilde\sigma^i$ and $\tilde b= \tilde b^i$. Under the assumptions of Theorem \ref{thm-LPS}, we deduce from Lemma \ref{lem-LPS} that $\d\nu^i_t= v^i_t\,\d x$ with $v^i\in L^\infty(0,T; L^1\cap L^{\bar p}(\R^d)),\, i=1,2$. Recall that $\bar p=\frac{p}{p-2} >\frac{p}{p-1}= p'$. We need the following estimates.

\begin{proposition}\label{prop-LPS-1}
There exists a constant $C>0$ such that, for all $t\in [0,T]$,
  $$\aligned \tilde\D_\delta(\mu^1_t, \mu^2_t)&\leq \tilde\D_{\delta/3}(\nu^1_t, \nu^2_t)+ \frac C\delta \|b^1-b^2\|_{L^q(L^p)} \|u^2\|_{L^\infty(L^{p'})},\\
  \tilde\D_\delta(\nu^1_t, \nu^2_t)&\leq \tilde\D_{\delta/3}(\mu^1_t, \mu^2_t)+ \frac C\delta \|b^1-b^2\|_{L^q(L^p)} \|u^2\|_{L^\infty(L^{p'})}.
  \endaligned$$
\end{proposition}

\begin{proof}
First, note that if $\pi\in \mathcal C(\nu^1_t, \nu^2_t)$, then $((\psi^1_t)^{-1}, (\psi^2_t)^{-1})_\#\pi \in \mathcal C(\mu^1_t, \mu^2_t)$. Indeed, for any $A\in \mathcal B(\R^d)$,
  $$\aligned \big[((\psi^1_t)^{-1}, (\psi^2_t)^{-1})_\#\pi\big](A\times\R^d)&= \int_{\R^d\times \R^d} \textbf{1}_{A\times\R^d}\big((\psi^1_t)^{-1}(x), (\psi^2_t)^{-1}(y)\big)\,\d\pi(x,y)\\
  &= \int_{\R^d} \textbf{1}_A ((\psi^1_t)^{-1}(x))\,\d \nu^1_t(x)\\
  & = \big[((\psi^1_t)^{-1})_\# \nu^1_t\big](A) = \mu^1_t(A),\endaligned$$
where the last equality follows from Lemma \ref{lem-LPS}. Similarly,
$$\big[((\psi^1_t)^{-1}, (\psi^2_t)^{-1})_\#\pi\big](\R^d\times A)=\mu^2_t(A).$$
Hence we get the desired result.

Now we have
  \begin{equation*}
  \aligned \tilde\D_\delta(\mu^1_t, \mu^2_t)&\leq \int_{\R^d\times\R^d} \log\bigg(1+ \frac{|x-y|^2}{\delta^2}\bigg)\d\big[((\psi^1_t)^{-1}, (\psi^2_t)^{-1})_\#\pi\big](x,y)\\
  &= \int_{\R^d\times\R^d} \log\bigg(1+ \frac{|(\psi^1_t)^{-1}(x)- (\psi^2_t)^{-1}(y)|^2}{\delta^2}\bigg)\d\pi(x,y).
  \endaligned
  \end{equation*}
As a result,
  \begin{equation}\label{prop-LPS-1.1}
  \aligned \tilde\D_\delta(\mu^1_t, \mu^2_t)&\leq \int_{\R^d\times\R^d} \log\bigg(1+ \frac{2|(\psi^1_t)^{-1}(x)- (\psi^1_t)^{-1}(y)|^2}{\delta^2}\bigg)\d\pi(x,y)\\
  &\hskip13pt +\int_{\R^d\times\R^d} \log\bigg(1+ \frac{2|(\psi^1_t)^{-1}(y)- (\psi^2_t)^{-1}(y)|^2}{\delta^2}\bigg)\d\pi(x,y)\\
  &=: J_1+J_2.\endaligned
  \end{equation}

First, by \eqref{prop-LPS.0} and \eqref{LPS-0},
  $$\aligned
  |(\psi^1_t)^{-1}(x)-(\psi^1_t)^{-1}(y)|&\leq |x-y|+ |\phi^1_t((\psi^1_t)^{-1}(x)) - \phi^1_t((\psi^1_t)^{-1}(y))|\\
  &\leq |x-y|+\frac12 |(\psi^1_t)^{-1}(x)-(\psi^1_t)^{-1}(y)|,\endaligned$$
which implies
  \begin{equation}\label{prop-LPS.1}
  |(\psi^1_t)^{-1}(x)-(\psi^1_t)^{-1}(y)|\leq 2|x-y|,\quad x,y\in \R^d.
  \end{equation}
Therefore,
  \begin{equation}\label{prop-LPS-1.2}
  J_1\leq \int_{\R^d\times\R^d} \log\bigg(1+ \frac{8|x-y|^2}{\delta^2}\bigg)\d\pi(x,y).
  \end{equation}
Next, by \eqref{prop-LPS.0.5}, $$\aligned
  J_2&\leq \int_{\R^d\times\R^d} \log\bigg(1+ \frac{8|\phi^1_t((\psi^2_t)^{-1}(y))-\phi^2_t((\psi^2_t)^{-1}(y))|^2}{\delta^2}\bigg)\d\pi(x,y)\\
  &= \int_{\R^d} \log\bigg(1+ \frac{8|\phi^1_t((\psi^2_t)^{-1}(y))-\phi^2_t((\psi^2_t)^{-1}(y))|^2}{\delta^2}\bigg)\d\nu^2_t(y)\\
  &= \int_{\R^d} \log\bigg(1+ \frac{8|\phi^1_t(y)-\phi^2_t(y)|^2}{\delta^2}\bigg)\d\mu^2_t(y).
  \endaligned$$
Using the simple inequality $\log(1+r^2)\leq 2\log(1+r)\leq 2r$, for any $r\geq 0$, we obtain
\begin{eqnarray*}J_2&\leq& \int_{\R^d} \frac6\delta |\phi^1_t(y)-\phi^2_t(y)|\,\d\mu^2_t(y) \leq \frac6\delta \|\phi^1_t-\phi^2_t\|_{L^p} \|u^2_t\|_{L^{p'}}\\
&\leq& \frac C\delta \|b^1-b^2\|_{L^q(L^p)} \|u^2\|_{L^\infty(L^{p'})},
\end{eqnarray*}
where the last inequality follows from Lemma \ref{lem-LPS-0}. Combining this inequality with \eqref{prop-LPS-1.1} and \eqref{prop-LPS-1.2}, we arrive at
  $$\tilde\D_\delta(\mu^1_t, \mu^2_t)\leq \int_{\R^d\times\R^d} \log\bigg(1+ \frac{9|x-y|^2}{\delta^2}\bigg)\d\pi(x,y) + \frac C\delta \|b^1-b^2\|_{L^q(L^p)} \|u^2\|_{L^\infty(L^{p'})}.$$
Taking infimum with respect to $\pi\in \mathcal C(\nu^1_t, \nu^2_t)$ yields the desired result.

The second inequality can be proved in a similar way, and it is simpler by noting that if $\hat\pi \in \mathcal C(\mu^1_t, \mu^2_t)$, then $(\psi^1_t, \psi^2_t)_\# \hat\pi \in \mathcal C(\nu^1_t, \nu^2_t)$.
\end{proof}

Now we are ready to present the

\begin{proof}[Proof of Theorem \ref{thm-LPS}]
Note that
  $$\nabla\tilde\sigma^i_t(x)= (\nabla^2\phi^i_t)((\psi^i_t)^{-1}(x))\, \nabla (\psi^i_t)^{-1}(x),\quad \nabla\tilde b^i_t(x)= \lambda (\nabla \phi^i_t)((\psi^i_t)^{-1}(x))\, \nabla (\psi^i_t)^{-1}(x).$$
Since $\nabla (\psi^i_t)^{-1}(x)$ and $\nabla \psi^i_t(x)$ are bounded uniformly in $(t,x)\in [0,T]\times \R^d$, it is clear that
  \begin{equation}\label{thm-LPS.1}
  \big\| \nabla \tilde \sigma^i\big\|_{L^2(L^p)}^2\leq C \|b^i\|_{L^q(L^p)}^2 \quad \mbox{and} \quad \big\| \nabla \tilde b^i\big\|_{L^1(L^p)}\leq C \|b^i\|_{L^q(L^p)}.
  \end{equation}
Moreover, it is easy to see that $\tilde b^i\in L^1(0,T; L^p(\R^d,\R^d))$. Therefore, the conditions of Theorem \ref{thm-variant} are verified with $p_1= p/2$ and $p_2=p$, except that $\tilde \sigma^i\in L^2(0,T; L^p(\R^d,\mathcal M_{d,d}))$. Indeed, $\tilde \sigma^i$ is uniformly bounded and non-degenerate. But this will not cause trouble, since by Remark \ref{2-rem}(3), it suffices to have $\int_0^T\! \int_{\R^d} \|\tilde\sigma^i_t\|^2 v^i_t\,\d x\d t<+\infty$, which is obvious due to $v^i\in L^\infty(0,T; L^1(\R^d))$. Thus, we have by Theorem \ref{thm-variant} that
  \begin{equation*}
  \begin{split}
  \tilde\D_\delta(\nu^1_t, \nu^2_t)&\leq \tilde\D_\delta(\nu^1_0, \nu^2_0) +C_1 \bigg(\frac1{\delta^2}\|\tilde\sigma^1- \tilde\sigma^2\|_{L^2(L^p)}^2 + \big\| \nabla \tilde \sigma^1\big\|_{L^2(L^p)}^2 \bigg)\\
  &\hskip13pt + C_2 \bigg(\frac1{\delta}\|\tilde b^1- \tilde b^2\|_{L^1(L^p)} + \big\| \nabla \tilde b^1\big\|_{L^1(L^p)} \bigg).
  \end{split}
  \end{equation*}
Combining the estimates in \eqref{thm-LPS.1} with Proposition \ref{prop-LPS}, we obtain
  $$\aligned
  \tilde\D_\delta(\nu^1_t, \nu^2_t)&\leq \tilde\D_\delta(\nu^1_0, \nu^2_0) +C'_1 \bigg(\frac1{\delta^2}\|b^1- b^2\|_{L^q(L^p)}^2 + \| b^1\|_{L^q(L^p)}^2 \bigg)\\
  &\hskip13pt + C'_2 \bigg(\frac1{\delta}\|b^1- b^2\|_{L^q(L^p)} + \| b^1\|_{L^q(L^p)} \bigg). \endaligned$$
Finally, we complete the proof by applying Proposition \ref{prop-LPS-1}.
\end{proof}

It remains to prove Theorem \ref{thm-LPS-Wass}. To this end, we first prove the following assertion.

\begin{proposition}\label{prop-LPS-Wass}
Let $\alpha\in (2, p)$. For each $i=1,2$, let $\nu^i_t= v^i_t(x)\,\d x$ be the solution to \eqref{FPE-LPS-1} satisfying $v^i\in L^\infty \big(0,T; L^1\cap L^{p/(p-\alpha)}(\R^d)\big)$. Assume that $\nu^i_0$ has finite moment of order $2\alpha-2,\, i=1,2$. Then there exists  some positive constant $\tilde C_\alpha$ such that for all $t\in[0,T]$,
  $$W_2(\nu^1_t, \nu^2_t)\leq \tilde C_\alpha\Big[W_\alpha(\nu^1_0, \nu^2_0) + \|v^2\|_{L^\infty(L^{p/(p-\alpha)})}^{1/\alpha} \big(\|\tilde b^1- \tilde b^2\|_{L^\alpha (L^p)} + \|\tilde \sigma^1- \tilde \sigma^2\|_{L^\alpha (L^p)}\big) \Big].$$
\end{proposition}

\begin{proof}
Since $\alpha>2$, we have $2\alpha-2>\alpha$, thus the $\alpha$-th moments of $\nu^1_0$ and  $\nu^2_0$ are finite. Take $\pi_\alpha\in \mathcal C(\nu^1_0, \nu^2_0)$ such that
  $$W_\alpha(\nu^1_0, \nu^2_0)^\alpha =\int_{\R^d\times \R^d} |x-y|^\alpha\,\d \pi_\alpha(x,y).$$
Analogous to the beginning of the proof of Theorem \ref{2-thm-1}, we can find a probability space $(\Omega, \mathcal G, P)$ on which there are defined two stochastic processes $Y^1_t$ and $Y^2_t$ and a Brownian motion $W_t$, such that $\pi_\alpha={\rm law}(Y^1_0,Y^2_0)$; moreover, for $i=1,2$, $\nu^i_t= {\rm law}(Y^i_t)$ and
  $$Y^i_t=Y^i_0+\int_0^t \tilde b^i_s(Y^i_s)\,\d s  +\int_0^t \tilde\sigma^i_s(Y^i_s)\,\d W_s,\quad \mbox{for all } t\in[0, T].$$
Following the proof of Lemma \ref{2-lem-growth}(1) and using the fact that $\tilde b^i$ and $\tilde\sigma^i$ are uniformly bounded on $[0,T]\times\R^d$, we have
  $$\sup_{0\leq t\leq T} \int_{\R^d} |x|^{2\alpha-2}\,\d\nu^i_t(x) =\sup_{0\leq t\leq T} \E|Y^i_t|^{2\alpha -2} < \infty,\quad i=1,2.$$

Let $Z_t= Y^1_t- Y^2_t$. For $\alpha\in (2, p)$, by It\^o's formula,
  \begin{equation*} \aligned
  \d|Z_t|^\alpha &= \alpha|Z_t|^{\alpha-2}\<Z_t, (\tilde \sigma^1_t(Y^1_t)- \tilde \sigma^2_t(Y^2_t))\,\d W_t\> +\alpha|Z_t|^{\alpha-2}\<Z_t, \tilde b^1_t(Y^1_t)- \tilde b^2_t(Y^2_t) \>\,\d t\\
  &\hskip13pt + \frac \alpha2 |Z_t|^{\alpha-2} \|\tilde \sigma^1_t(Y^1_t)- \tilde \sigma^2_t(Y^2_t)\|^2\,\d t +\alpha(\alpha-2)|Z_t|^{\alpha-4} |(\tilde \sigma^1_t(Y^1_t)- \tilde \sigma^2_t(Y^2_t))^\ast Z_t|^2\,\d t\\
  &=: \d M_t + \d I_1(t) + \d I_2(t) +\d I_3(t).
  \endaligned
  \end{equation*}
Since
  $$\aligned \E \<M\>_T &= \alpha^2 \E\int_0^T  |Z_t|^{2\alpha-4} |(\tilde \sigma^1_t(Y^1_t)- \tilde \sigma^2_t(Y^2_t))^\ast Z_t|^2\,\d t\\
  &\leq \alpha^2\big(\|\tilde \sigma^1\|_{L^\infty} + \|\tilde \sigma^2\|_{L^\infty}\big)^2\, \E\int_0^T  |Z_t|^{2\alpha-2}\,\d t <\infty, \endaligned$$
which implies that $M_t$ is a square integrable martingale. We have
  \begin{equation}\label{prop-LPS-Wass.1}
  \aligned
  \d I_1(t)&= \alpha |Z_t|^{\alpha-2} \<Z_t, \tilde b^1_t(Y^1_t)- \tilde b^1_t(Y^2_t) \>\,\d t+ \alpha |Z_t|^{\alpha-2} \<Z_t, \tilde b^1_t(Y^2_t)- \tilde b^2_t(Y^2_t) \>\,\d t\\
  &\leq \alpha L |Z_t|^\alpha \,\d t + \alpha |Z_t|^{\alpha-1} |\tilde b^1_t(Y^2_t)- \tilde b^2_t(Y^2_t)|\,\d t,
  \endaligned
  \end{equation}
where $L$ is the Lipschitz constant of $\tilde{b}^1_t$ which is independent of $t\in [0,T]$ (see e.g. \cite[Proposition 4.3]{Flandoli}). Next,
  $$\aligned
  \d I_2(t) +\d I_3(t)&\leq  \Big(\frac \alpha2 + \alpha(\alpha-2)\Big)|Z_t|^{\alpha-2} \|\tilde \sigma^1_t(Y^1_t)- \tilde \sigma^2_t(Y^2_t)\|^2\,\d t\\
  & \leq 2\alpha^2|Z_t|^{\alpha-2} \big(\|\tilde \sigma^1_t(Y^1_t)- \tilde \sigma^1_t(Y^2_t)\|^2 + \|\tilde \sigma^1_t(Y^2_t)- \tilde \sigma^2_t(Y^2_t)\|^2\big)\,\d t.
  \endaligned$$
Denote by
  $$\d A_t = \frac{\|\tilde \sigma^1_t(Y^1_t)- \tilde \sigma^1_t(Y^2_t)\|^2}{|Y^1_t- Y^2_t|^2}\,\d t$$
with the convention that $\frac 00=0$. According to 
\cite[Lemma 7]{Flandoli11} (see also \cite[Lemma 4.5]{Flandoli}),
  \begin{equation}\label{prop-LPS-Wass.1.5}
  \E e^{kA_t} <\infty\quad \mbox{for any } k\in\R.
  \end{equation}
Using this notation, we have
  \begin{equation}\label{prop-LPS-Wass.2}
  \d I_2(t) +\d I_3(t)\leq 2\alpha^2 |Z_t|^{\alpha}\,\d A_t + 2\alpha^2 |Z_t|^{\alpha-2} \|\tilde \sigma^1_t(Y^2_t)- \tilde \sigma^2_t(Y^2_t)\|^2\,\d t.
  \end{equation}
Combining \eqref{prop-LPS-Wass.1} and \eqref{prop-LPS-Wass.2} yields
\begin{eqnarray*}
\d|Z_t|^\alpha&\leq& \d M_t + |Z_t|^{\alpha} \d\tilde A_t + \alpha |Z_t|^{\alpha-1} |\tilde b^1_t(Y^2_t)- \tilde b^2_t(Y^2_t)|\,\d t \\
&&+ 2\alpha^2 |Z_t|^{\alpha-2} \|\tilde \sigma^1_t(Y^2_t)- \tilde \sigma^2_t(Y^2_t)\|^2\,\d t,
\end{eqnarray*}
where $\tilde A_t= \alpha L t + 2\alpha^2 A_t$. By Young's inequality,
  $$|Z_t|^{\alpha-1} |\tilde b^1_t(Y^2_t)- \tilde b^2_t(Y^2_t)| \leq \frac{\alpha-1}\alpha |Z_t|^{\alpha} +\frac1\alpha |\tilde b^1_t(Y^2_t)- \tilde b^2_t(Y^2_t)|^\alpha$$
and
  $$|Z_t|^{\alpha-2} \|\tilde \sigma^1_t(Y^2_t)- \tilde \sigma^2_t(Y^2_t)\|^2 \leq \frac{\alpha-2}\alpha |Z_t|^{\alpha} +\frac2\alpha \|\tilde \sigma^1_t(Y^2_t)- \tilde \sigma^2_t(Y^2_t)\|^\alpha.$$
Therefore,
  $$\d|Z_t|^\alpha\leq \d M_t + |Z_t|^{\alpha} \d\hat A_t + |\tilde b^1_t(Y^2_t)- \tilde b^2_t(Y^2_t)|^\alpha\,\d t + 4\alpha \|\tilde \sigma^1_t(Y^2_t)- \tilde \sigma^2_t(Y^2_t)\|^\alpha\,\d t,$$
where
  $$\hat A_t= \tilde A_t + (2\alpha^2 -3\alpha -1)t= 2\alpha^2 A_t + C_{\alpha, L} t\geq0$$
with $C_{\alpha, L}=2\alpha^2+(L-3)\alpha-1$. As a result,
  $$\d\big(e^{-\hat A_t} |Z_t|^\alpha\big)\leq e^{-\hat A_t}\,\d M_t +e^{-\hat A_t}|\tilde b^1_t(Y^2_t)- \tilde b^2_t(Y^2_t)|^\alpha\,\d t +4\alpha e^{-\hat A_t} \|\tilde \sigma^1_t(Y^2_t)- \tilde \sigma^2_t(Y^2_t)\|^\alpha\,\d t.$$
Notice  
that $M_t$ is a square integrable martingale and $e^{-\hat A_t}\leq 1$; hence
  $$\E\big(e^{-\hat A_t} |Z_t|^\alpha\big) \leq \E|Z_0|^\alpha + \E\int_0^t e^{-\hat A_s}\big(|\tilde b^1_s(Y^2_s)- \tilde b^2_s(Y^2_s)|^\alpha + 4\alpha \|\tilde \sigma^1_s(Y^2_s)- \tilde \sigma^2_s(Y^2_s)\|^\alpha\big)\,\d s.$$

By H\"older's inequality and \eqref{prop-LPS-Wass.1.5},
  $$\aligned
  \E|Z_t|^2 &\leq \big[\E\big(e^{-\hat A_t} |Z_t|^\alpha\big)\big]^{2/\alpha} \big[\E e^{\alpha\hat A_t/(\alpha-2)} \big]^{(\alpha-2)/\alpha}\\
  &\leq C_\alpha \bigg[\E|Z_0|^\alpha + \E\int_0^t e^{-\hat A_s}\big(|\tilde b^1_s(Y^2_s)- \tilde b^2_s(Y^2_s)|^\alpha + 4\alpha \|\tilde \sigma^1_s(Y^2_s)- \tilde \sigma^2_s(Y^2_s)\|^\alpha\big)\,\d s\bigg]^{2/\alpha}.
  \endaligned$$
Consequently,
  $$\aligned \|Z_t\|_{L^2(P)}&\leq \tilde C_\alpha\bigg[\|Z_0\|_{L^\alpha(P)} + \Big(\int_0^t\!\!\int_{\R^d} \big(|\tilde b^1_s- \tilde b^2_s|^\alpha + \|\tilde \sigma^1_s- \tilde \sigma^2_s\|^\alpha\big) v^2_s\,\d x\d s\Big)^{1/\alpha}\bigg]\\
  &\leq \tilde C_\alpha\bigg[\|Z_0\|_{L^\alpha(P)} + \Big(\int_0^t \big(\|\tilde b^1_s- \tilde b^2_s\|_{L^p}^\alpha + \|\tilde \sigma^1_s- \tilde \sigma^2_s\|_{L^p}^\alpha\big) \|v^2_s\|_{L^{p/(p-\alpha)}}\,\d s\Big)^{1/\alpha}\bigg]\\
  &\leq \tilde C_\alpha\Big[\|Z_0\|_{L^\alpha(P)} + \|v^2\|_{L^\infty(L^{p/(p-\alpha)})}^{1/\alpha} \big(\|\tilde b^1- \tilde b^2\|_{L^\alpha (L^p)} + \|\tilde \sigma^1- \tilde \sigma^2\|_{L^\alpha (L^p)}\big) \Big].
  \endaligned$$
This finishes the proof.
\end{proof}

New we can present the

\begin{proof}[Proof of Theorem \ref{thm-LPS-Wass}]
Under the assumptions, it is clear that the conditions of Proposition \ref{prop-LPS-Wass} are verified. Following the argument of Proposition \ref{prop-LPS-1}, we can show that for $\alpha\in [2,p\wedge q)$,
  $$\aligned W_\alpha(\mu^1_t, \mu^2_t)&\leq 2 W_\alpha(\nu^1_t, \nu^2_t) + 2 \|b^1-b^2\|_{L^q(L^p)} \|u^2\|_{L^\infty(L^{p/(p-\alpha)})}^{1/\alpha},\\
  W_\alpha(\nu^1_t, \nu^2_t)&\leq 2 W_\alpha(\mu^1_t, \mu^2_t) + 2 \|b^1-b^2\|_{L^q(L^p)} \|u^2\|_{L^\infty(L^{p/(p-\alpha)})}^{1/\alpha}.  \endaligned$$
Combining these inequalities with Propositions \ref{prop-LPS} and \ref{prop-LPS-Wass}, we complete the proof.
\end{proof}

\medskip

\noindent \textbf{Acknowledgements.} The first author is supported by the National Natural Science Foundation of China (Nos. 11401403, 11571347). The second author is supported by the National Natural Science Foundation of China (Nos. 11571347, 11688101), the Seven Main Directions (Y129161ZZ1) and the Special Talent Program of the Academy of Mathematics and Systems Science, Chinese Academy of Sciences.

\end{document}